\documentclass[A4paper,12pt]{amsart}
\usepackage[english]{babel}
 \usepackage[utf8]{inputenc}
\usepackage[T1]{fontenc}
\usepackage{amsmath, amssymb,mathtools}

\usepackage[a4paper,top=3cm,bottom=2cm,left=3cm,right=3cm,marginparwidth=1.75cm]{geometry}

\usepackage{amsthm}
\usepackage{amsfonts}
\usepackage{amssymb}
\usepackage{amsmath}
\usepackage{listings} % for code
\usepackage{xcolor} % for code
\usepackage{upquote} % for code (apostrophe in verbatim)
\usepackage{faktor}
\usepackage{graphicx}
\usepackage{csquotes}
\usepackage{enumitem}
\usepackage[colorinlistoftodos]{todonotes}
\usepackage[normalem]{ulem}
\usepackage[colorlinks=true, allcolors=blue]{hyperref}
\usepackage{url}
\usepackage[style=alphabetic,hyperref,bibencoding=latin1, giveninits=true,isbn=false, backend=bibtex, maxnames=10,,doi=false,url=false,eprint=false]{biblatex}
\usepackage{etoc}

\newtheorem{theorem}{Theorem}
[section]
\newtheorem*{theorem*}{Theorem}

\newtheorem{corollary}[theorem]{Corollary}
\newtheorem{lemma}[theorem]{Lemma}
\newtheorem{proposition}[theorem]{Proposition}

\theoremstyle{remark}
\newtheorem{definition}[theorem]{Definition}
\newtheorem{remark}[theorem]{Remark}
\newtheorem{question}[theorem]{Question}

\def\NN{\mathbb{N}}
\def\QQ{\mathbb{Q}}

\def\ZZ{\mathbb{Z}}

\usepackage{listings}
\usepackage{xcolor}

\definecolor{codegreen}{rgb}{0,0.6,0}
\definecolor{codegray}{rgb}{0.5,0.5,0.5}
\definecolor{codepurple}{rgb}{0.58,0,0.82}
\definecolor{backcolour}{rgb}{0.95,0.95,0.92}

\lstdefinestyle{sagestyle}{
    backgroundcolor=\color{backcolour},   
    commentstyle=\color{codegreen},
    keywordstyle=\color{magenta},
    numberstyle=\tiny\color{codegray},
    stringstyle=\color{codepurple},
    basicstyle=\ttfamily\footnotesize,
    breakatwhitespace=false,         
    breaklines=true,                 
    captionpos=b,                    
    keepspaces=true,                 
    numbers=none,                    
    numbersep=5pt,                  
    showspaces=false,                
    showstringspaces=false,
    showtabs=false,                  
    tabsize=2
}

\usepackage[utf8]{inputenc}
\usepackage{amsmath, amssymb,enumitem}
\usepackage{todonotes}

\usepackage{bm}

\newcommand{\sn}{\ell_n+u_n}
\newcommand{\tn}{\ell_n+u_n}%for AB thm.2.4

\title{If a machine did it, it is probably transcendental (even $p$-adically)}
\author{Laura Capuano}
\address[Laura Capuano]{Dipartimento di Matematica e Fisica, Univ.\ degli Studi Roma Tre}
\email{laura.capuano@uniroma3.it}
\author{Sara Checcoli}
\address[Sara Checcoli]{Univ.\ Grenoble Alpes, CNRS, IF, 38000 Grenoble, France}
\email{sara.checcoli@univ-grenoble-alpes.fr}
\author{Marzio Mula}
\address[Marzio Mula]{Research Institute CODE, Univ.\ of the Bundeswehr Munich}
\email{marzio.mula@unibw.de}
\author{Lea Terracini}
\address[Lea Terracini]{Dipartimento di Informatica, Univ.\ di Torino}
\email{lea.terracini@unito.it}

\date{\today}

\begin{document}

\begin{abstract}
Continued fraction expansions provide a well-established bridge between algebraic properties of numbers and combinatorics on words. In this article, we investigate the algebraicity of $p$-adic numbers whose continued fractions arise from certain classes of words which generalize the classical automatic, periodic and palindromic words.
Our main result shows that, under mild conditions on the $p$-adic continued fraction expansion, such numbers are either algebraic of degree at most 2 or transcendental.
This result provides an analogue of results of Bugeaud and Adamczewski-Bugeaud in the real setting and extends previous works that were limited to specific choices of $p$-adic floor functions and less general classes of words.
\end{abstract}
\maketitle
\section{Introduction}
Continued fractions can be viewed as a way to map numbers in a given set to words, i.e.\ possibly infinite sequences of elements from a specified alphabet $\mathcal{A}$. 

There are several ways to construct continued fraction algorithms depending on the context. In this article, we will focus on the case where the numbers belong to a field $K$, which is either $\mathbb{R}$ or $\mathbb{Q}_p$, and where the alphabet $\mathcal{A}$ is a subset of either $\mathbb{Z}$ or $\mathbb{Z}[1/p]$. We denote by $\NN$ the set of strictly positive integers.

Most classical constructions are based on the choice of a function, called \emph{floor function}, $s: K \to K$, which determines the choice of the alphabet $\mathcal{A}=\mathrm{Im}(s)$. 

In the archimedean case, we have $K = \mathbb{R}$, $\mathcal{A}\subset \mathbb{Z}$ and $s$ is the classical floor function. However, in the $p$-adic setting where $K=\mathbb{Q}_p$ and $\mathcal{A}\subset \mathbb{Z}[1/p]$, there are infinitely many possible choices for the floor function $s$ and they all share analogous properties to the classical one (see Section \ref{subsec-p-adic floor-function}). 

In all cases, given a floor function $s$ and $\alpha \in K$, the corresponding word $a_0 a_1 a_2 \cdots$ is obtained via the algorithm:
\begin{equation}\label{CF-algo-floor-fun-def}
\gamma_0 = \alpha, \quad a_n = s(\gamma_n), \quad \gamma_{n+1} = \frac{1}{\gamma_n - a_n} \quad \text{if } \gamma_n \neq a_n,
\end{equation}
which stops when $\gamma_n = a_n$.  The $a_n$'s are called the \emph{partial quotients} of $\alpha$. From these, one can define the two important sequences of \emph{continuants} $ (A_n)_{n\geq-1}$ and $ (B_n)_{n\geq -1}$, given by the following recurrences :
\begin{equation}
    \label{def-continuants}
\begin{aligned}
A_{-1}&= 1, \quad A_0= a_0, \quad A_n= a_n A_{n-1} + A_{n-2}, \\
B_{-1}&= 0, \quad B_0= 1,  \quad B_n= a_n B_{n-1} + B_{n-2}.
\end{aligned}
\end{equation}
The quotient $A_n/B_n$ is called the  \emph{$n$-th convergent} of the continued fraction expansion (with respect to $s$) of $\alpha$. Indeed, it is a classical result that the sequence $(A_n/B_n)_n$ converges $p$-adically to $\alpha$~\cite[Theorem 2]{Browkin1978}.

If a number $\alpha$ corresponds, via a continued fraction algorithm with floor function $s$, to the word $a_0 a_1 a_2 \cdots$, we write \[\alpha = [a_0, a_1, a_2, \ldots]_s.\] When $K = \mathbb{R}$ and $s$ is the classical floor function, we simply write $\alpha = [a_0, a_1, a_2, \ldots]$.

Continued fractions share several features with the more common base-$b$ representations of numbers~\cite{Berthé_Rigo_2010}: they can be computed algorithmically, they are well-suited for (some) arithmetic operations, and, surprisingly, they carry over some properties from the realm of word-combinatoric to the realm of numbers. In this article we will investigate the latter feature in the case of $p$-adic continued fractions.
\par Before doing so, let us briefly consider the classical archimedean case, from which our approach is inspired. 
It is well known~\cite[Ch.\,4]{angell2021irrationality} that archimedean continued fractions provide a bijection between real numbers in $[0,1)$ and words on $\mathbb{N}$ (excluding finite words ending with~$1$). Not only does this bijection restrict to a bijection between rationals and finite words, but, as already shown by Lagrange, it also restricts to a bijection between algebraic numbers of degree $2$ and eventually periodic words. This motivated the following:
\begin{question}\label{question-intro}
    Are there families of  \emph{special} words that correspond to families of algebraic numbers of degree $>2$?
\end{question}
 Notably, the answer to this question is expected to be negative. The reasoning behind this expectation can be seen as a sort of statement in \emph{unlikely intersections}: any set of words exhibiting some form of structure constitutes a small subset within the set of all words. Consequently, words from such a set arise from the continued fraction expansion of a  small subset of the real numbers. Such a small set is unlikely to intersect the set of algebraic numbers, which is itself small, unless the words we took were \emph{very special} i.e.\ finite or periodic.
 
Historically, two families of special words have mainly been considered as natural generalizations of periodic words: words on finite subsets of $\mathbb{N}$ and words containing increasingly large repeated or mirrored prefixes. Although these are the primary candidates for being new special families with algebraic continued fractions, all available partial results suggest they fall short of this goal, revealing instead a substantial gap between periodic words and any form of generalization.

Indeed, the former family, i.e.\ words on a finite alphabet, was already studied by Borel, who proved that it is in bijection with a measure-zero set of real numbers.
As mentioned in~\cite[§4]{shallit1992real}, it seems that Khintchine was the first to ask whether this set contains any algebraic number of degree~$>2$. While this question remains open in general, in the meantime several specific cases have been studied and suggest that the answer should be negative. Such known partial results include families of quasi-periodic words~\cites[Ch.\,VII]{maillet1906introduction}{baker1962:cfTransc}, and some words of linear complexity i.e.\ words containing only $O(n)$ distinct subwords of length $n$, such as Thue-Morse words~\cites{queffelec1998:thueMorse}{adamBug2007:thueMorse} or Sturmian words   ~\cites{queffelec2000irrational}{ALLOUCHE200139}.

In a series of later works ~\cites{adamBug2005:complIIContFrac}{AdamczewskiBugeaudDavison06}{AdamczewskiBugeaud2007:mailletBaker}{AdamczewskiBugeaud2007:palindromic}{Bugeaud2010AutomaticCF} the authors considered generalisations of quasi-periodic and quasi-palidromic words, namely the \emph{$\spadesuit$ words} and \emph{$\clubsuit$ words}. 

Before giving their definitions, we fix some notation: for a finite word $\omega$, we denote its length by $|\omega|$ and we denote by $\widetilde{\omega}$ the word obtained by reversing the order of the letters of $\omega$.
According to \cite[beginning of Sections 3 and 5]{Bugeaud2010AutomaticCF}), we have the following:
\begin{definition}\label{def-spade-club-with-constant}
Let $c>0$ be a real number. A word $\omega$ has \emph{property $\spadesuit$} (respectively \emph{property $\clubsuit$}) \emph{with constant $c$} if it is not ultimately periodic and if there exists three sequences $(U_n)_n, (V_n)_n, (W_n)_n$ of finite words such that:
\begin{enumerate}[label=(\roman*)]
\item\label{it1} for every $n\geq 1$, $W_n U_n V_n U_n$ (respectively, $W_n U_n V_n \widetilde{U_n}$)  is a prefix of $\omega$ ;
    \item\label{it2} for every $n\geq 1$, $\max(|V_n|/|U_n|, |W_n|/|U_n|)\leq c$;
        \item\label{it3} the sequence $(|U_n|)_n$ is increasing.
\end{enumerate}    
\end{definition}  Notably,  $\spadesuit$ words include all linear complexity words~\cite[Theorem.\,1.1]{Bugeaud2010AutomaticCF}. In particular, thanks to a theorem by Cobham~\cite{cobham68:tagMachinesComplexity}, they include all automatic words, which are words on a finite alphabet that can be generated by a finite state machine. However, both classes of $\spadesuit$ words and $\clubsuit$ words include also words on infinite alphabets  (see Section \ref{section-2-picche-fiori-esempi} for more details and examples). 

In 
\cite{adamBug2005:complIIContFrac, AdamczewskiBugeaudDavison06, AdamczewskiBugeaud2007:palindromic} the authors showed that, if $\omega=a_1a_2\cdots$ has property $\spadesuit$ or $\clubsuit$, and assuming the extra condition that the limit superior of the sequence $({|W_n|}/{|U_n|})_n$ is \emph{small enough}, then the real number $[0, a_1, a_2, \dots]$ is transcendental. This was generalised by Bugeaud \cite[Theorem 1.3] {Bugeaud2010AutomaticCF} to all $\spadesuit$ and $\clubsuit$ words as it follows:
\begin{theorem}[Bugeaud]\label{thm-bugeaud}
\label{thm:Bugeaud}
    Let $\alpha=[0,a_1,a_2,\dots]$, with $a_i\in \mathbb{Z}_{> 0}$, and let $A_n/B_n$ be the $n$-th convergent of its continued fraction expansion. 
    Suppose that the word $a_1 a_2 a_3 \cdots$ satisfies condition $\spadesuit$ or $\clubsuit$, and that there exist a real $C>1$ such that $|B_n|\leq C^n$.
    Then, $\alpha$ is either quadratic or transcendental.
\end{theorem}

In this article, we consider the analogue of Theorem \ref{thm-bugeaud} in the $p$-adic setting, where $p$ is any odd prime. A first naive attempt could be to simply replace $\mathbb{N}$ with $\mathbb{Z}[1/p]$ and $\mathbb{R}$ with $\mathbb{Q}_p$. 
However, this translation is far from straightforward for various reasons.

A structural key difficulty in the $p$-adic case lies in the absence of a \emph{canonical} floor function $s:\mathbb{Q}_p\rightarrow \mathbb{Q}$. As we already noticed, there are infinitely many possible choices for $s$, and each leads to a distinct construction of $p$-adic continued fractions. In the last decades, many constructions of $p$-adic floor functions have been proposed, such as the classical ones by Ruban \cite{Ruban1970} and Browkin \cite{Browkin1978}.
However, none of them gives approximants as good as in the real case and all of them lack complete analogues of fundamental results from the real setting, such as Lagrange’s theorem. More precisely, it is either not true or unknown, in general, that all rational numbers arise from finite words (except for very specific choices of $s$), nor that all algebraic numbers of degree 2 arise from periodic words.

These differences make Question \ref{question-intro} even more challenging in this $p$-adic setting. The only known results in this context, provide negative answers to Question \ref{question-intro} only for very specific choices of floor functions and very special sets of words (see for instance \cites{ooto2017:padicTrasc}{longhi2024heights}). 

In this article, we provide, under some mild conditions, a negative answer to  Question \ref{question-intro} for $\spadesuit$ words and $\clubsuit$ words for any floor function, generalising most of the previous known results. More precisely, our main result is the following $p$-adic version of Theorem \ref{thm-bugeaud}:
\begin{theorem}\label{main-thm-intro}
     Let $p$ be an odd prime, let $\alpha\in \mathbb{Q}_p$ and let $s$ be a $p$-adic floor function. Suppose that $\alpha=[0,a_1,a_2,\dots]_s$ and let $A_n/B_n$ be the $n$-th convergent of this continued fraction expansion.  
    Suppose that the word $a_1 a_2 a_3 \cdots$ satisfies either condition $\spadesuit$ or condition $\clubsuit$ with constant $c>0$ and that for every $n\geq 1$: 
    \begin{enumerate}[label=(\roman*)]
        \item\label{bound-cont-abs-val} the sequences $(|A_n|_{\infty}^{1/n})_n, (|B_n|_{\infty}^{1/n})_n, (|B_n|_p^{1/n})_n$ are bounded;
        \item 
$|a_n|_p \geq p^k$ where $k$ is an explicit constant depending on $p$, $c$ and the bound for the sequence $\left(\max\left(|A_n|_{\infty}^{1/n}, |B_n|_{\infty}^{1/n}\right)\right)_n$.
\end{enumerate}
    Then, $\alpha$ is either rational, quadratic, or transcendental.
\end{theorem}
A more precise version of Theorem \ref{main-thm-intro} with explicit constants can be found in Section \ref{sec:main-res-examples} (more precisely, see Theorem  \ref{thm:pAdicBugeaud} for $\spadesuit$ words and Theorem \ref{main-thm-fiori} for $\clubsuit$ words). As a first corollary, we have the following result (proved in Section~\ref{subsec:applicationBrRub}) concerning the classical Ruban's and Browkin's continued fractions (see Section \ref{sub:Browkin_Ruban_floor} for more details):
\begin{corollary}
\label{cor:browRub}
    Let $p$ be an odd prime and let $s$ be a $p$-adic floor function. Let $\alpha=[0,a_1,a_2,\dots]_s$ and suppose that the sequence $(|B_n|_p^{1/n})_n$ is bounded, where $B_n$ is defined in \eqref{def-continuants}. Moreover, suppose that one of the following holds:
    \begin{enumerate}[label=(\alph*)]
        \item \label{itm:rep} the word $a_1a_2a_3 \cdots$ starts with arbitrarily long repetitions and either
    \begin{itemize}
     \item $s$ is Ruban's floor function, $\alpha$ is irrational and $|a_n|_p \geq p^4$ for every $n$, or
            \item $s$ is Browkin's floor function and $|a_n|_p \geq p^3$ for every $n$.
        \end{itemize}
         \item \label{itm:palin} the word $a_1a_2a_3 \cdots$ starts with arbitrarily long palindromes and either
        \begin{itemize}
        \item $s$ is Ruban's floor function, $\alpha$ is irrational and $|a_n|_p \geq p^2$ for every $n$, or
            \item $s$ is Browkin's floor function.
        \end{itemize}
    \end{enumerate}
     Then, $\alpha$ is either quadratic or transcendental.
\end{corollary}

In another direction, we also have the following corollary (proved in Section~\ref{applfinitealphabet}) on continued fractions arising from automatic sequences: 
\begin{corollary} \label{cor:largep} Let $n,m$ be non zero distinct integers. Let $\omega=a_1 a_2\cdots$ be an automatic word on the alphabet $\mathcal{A}=\{\frac{n}{p}, \frac{m}{p}\}$ where $p$ is an odd prime.  Then, for every $p$ large enough, there is a $p$-adic floor function $s$ such that $\mathcal{A}\subset \mathrm{Im}(s)$ and the number $[0,a_1,a_2,\ldots]_s$ is either rational, quadratic or transcendental.
\end{corollary}

Our proof strategy for Theorem \ref{main-thm-intro} is an adaption of that of Bugeaud \cite[Theorem 1.3]{Bugeaud2010AutomaticCF}. There, the main tool used is Schmidt's subspace theorem, which is a generalization of fundamental results by Liouville, Thue, and Roth, on the approximation of algebraic numbers by means of sequences of rational numbers (we refer to~\cite[Ch.\,6-7]{BombieriGubler2006} for an overview). The basic idea is that, if a number is approximated \emph{too well} by infinitely many  rational numbers, then it must be transcendental. This already gives an intuition, in the real case, of why words containing arbitrarily large repeated prefixes fail to yield an algebraic number, given that infinitely many of its convergents are \emph{good approximations} of some periodic words (i.e.\ algebraic numbers of degree $2$).
In our non-archimedean world, this general philosophy still applies, but one needs some extra work and a $p$-adic version of the subspace theorem proved by Schlikewei \cites{Schlick1,Schlick2} (see Section \ref{sec:p-adic-subspace} for details).

More specifically, the main idea underlying the proof of Theorem \ref{main-thm-intro} is to leverage the $\spadesuit$ (resp.\ $\clubsuit$) structure of $\alpha$ to approximate it by a sequence of quadratic (resp.\ rational) numbers $\alpha_n$ having a periodic (resp.\ palindromic) continued fraction expansion of the form $[W_n\overline{U_n V_n}]_s$ (resp.\ $[W_n U_n V_n\widetilde{U_n}]_s$). Then $\alpha$ and $\alpha_n$ have the same $|W_n|+2|U_n|+|V_n|$-th convergent, so that the quality of approximation is very good relatively to the arithmetic complexity (more precisely, the Weil height, cf.\ Definition~\ref{def:height}) of $\alpha_n$ which only depends on the $|W_n|+|U_n|+|V_n|$-th continuants. Thus, assuming $\alpha$ algebraic, one can construct suitable linear forms and a sequence of vectors depending on $\alpha_n$ that make the $p$-adic Subspace Theorem (Theorem \ref{thm:subspace-appl}) applicable. Then one gets a linear relation between some continuant of the continued fraction for $\alpha$. Finally, a brute force approach case by case, allows to transform it into a linear relation between $1,\alpha,\alpha^2$.

The structure of our article is as follows. In Section \ref{subsec:complexity}, 
we recall a result of Bugeaud linking word complexity and condition $\spadesuit$. In Section \ref{sec:main-res-examples}, we present explicit versions of Theorem \ref{main-thm-intro} (namely, Theorem \ref{thm:pAdicBugeaud} for $\spadesuit$ words and Theorem \ref{main-thm-fiori} for $\clubsuit$ words) along with applications, including proofs of Corollary \ref{cor:browRub} and Corollary \ref{cor:largep}. Section \ref{sec:p-adic-subspace} provides a statement of the $p$-adic Subspace Theorem, a key ingredient in our proofs. In Section \ref{sec:proof-picche} we prove Theorem  \ref{thm:pAdicBugeaud} and in Section \ref{sec:proof-fiori} we prove Theorem \ref{main-thm-fiori}. Finally, in Appendix \ref{section-2-picche-fiori-esempi}, we examine $\spadesuit$ and $\clubsuit$ words in detail, exploring their connections with well-known classes of words, such as those arising from automatic and Sturmian sequences, and providing concrete examples.

\section{Some results on words complexity}
\label{subsec:complexity}

The goal of this short section is to recall a result of Bugeaud \cite{Bugeaud2010AutomaticCF} relating word complexity and condition $\spadesuit$.

Recall that \emph{word} $\omega$ is a (possibly infinite) string of symbols  from a non-empty set $\mathcal{A}$, called \emph{alphabet}.
For a finite word $\omega$, we recall that we denote its length by $|\omega|$ and we denote by $\widetilde{\omega}$ the word obtained by reversing the order of the letters of $\omega$. We also denote by $\overline{\omega}$ the infinite purely periodic word of period $\omega$.

The \emph{complexity function} of a word $\omega=a_1a_2a_3 \cdots$ is the function $p_{\omega}:\mathbb{N}\rightarrow \mathbb{N}$ defined as 
\begin{equation*}
    p_{\omega}(n) =\# \{a_{\ell+1} \dots a_{\ell + n} \mid \ell \geq 0\},
\end{equation*}
that is,  $p_{\omega}(n)$ is the number of distinct blocks of length $n$ appearing in the writing of $\omega$. 

For example, the periodic word $\omega=\overline{011}$ has $p_{\omega}(1)=2$, $p_{\omega}(n)=3$ for all $n\geq 2$. More generally, a classical result in ~\cite[Thm.\,7.3]{morseHedlund:symbolicDyn1} shows that either $\omega$ is ultimately periodic, in which case the function $p_{\omega}$ is bounded, or $p_{\omega}(n)\geq n+1$ for all $n\geq 1$.

In view of this result and linking to our study of continued fraction, we have that, if the alphabet is $\NN$, the word complexity of an infinite word $\omega=a_1a_2a_3 \cdots$ gives information on the nature of the real number $\alpha=[0,a_1,a_2,\ldots]$. 

Indeed,  if $p_{\omega}(n)\leq n$ for some $n\geq 1$, then $\omega$ is ultimately periodic, hence, by Lagrange theorem, $\alpha$ is algebraic of degree 2.
 On the other hand, Bugeaud \cite[Theorem 1.1]{Bugeaud2010AutomaticCF} proves that if $\omega$ is not ultimately periodic and $\alpha$ is algebraic, then $p_{\omega}(n)$ cannot grow \emph{too slowly} i.e.\ one must have
\begin{equation}\label{eq-bugeaud-complexity}
    \lim_{n \rightarrow \infty} \frac{p_{\omega}(n)}{n}=\infty.
\end{equation}
A first step in his proof is to show that \emph{low complexity} words satisfy condition $\spadesuit$. More precisely (see \cite[Proof of Theorem 1.1]{Bugeaud2010AutomaticCF}):
\begin{theorem}\label{bugeaud-complexity-picche}
    Let $(a_i)_{i\geq 1}$ be a sequence of strictly positive integers which is not ultimately periodic and consider the word $\omega=a_1 a_2 a_3 \cdots$. Suppose that there exists an infinite set $\mathcal{N}\subset \mathbb{N}$ and a positive real $C>0$ such that $p_{\omega}(n)\leq n C$ for all $n\in\mathcal{N}$. Then $\omega$ satisfies condition $\spadesuit$ with constant $c=3C+1$. 
\end{theorem}
To the best of the authors' knowledge, there is no direct analogue of Theorem~\ref{bugeaud-complexity-picche} for words satisfying condition $\clubsuit$. In particular, it remains unclear whether this set contains specific complexity classes.

\section{Toolbox for the \texorpdfstring{$p$}{p}-adic case}\label{sec-toolbox-p-adic}

In this section we recall some basic properties of the $p$-adic continued fraction algorithms that will be used later.
In what follows, we will fix a prime $p>2$ and we will denote by $|\cdot|_p$ and $|\cdot|_{\infty}$  the $p$-adic and the Euclidean norm respectively and by $v_p$ the usual $p$-adic valuation.

\subsection{\texorpdfstring{$p$}{p}-adic floor functions}\label{subsec-p-adic floor-function}

There are several ways to construct continued fraction algorithms over $\QQ_p$, and several authors as Mahler \cite{Mahler1940}, Schneider \cite{Schneider1970}, Ruban \cite{Ruban1970} and Browkin \cite{Browkin1978, Browkin2000}  proposed different algorithms in the attempt of achieving the same properties which hold in the real case. All these constructions are based on the choice of a so called \textit{$p$-adic floor function} $s: \QQ_p \rightarrow \ZZ[1/p]$. We give the precise definition (which can also be found in \cite[Section 1]{Browkin1978}). 
\begin{definition}
\label{def:pFloor}
A \emph{$p$-adic floor function} is any function $s \colon \mathbb{Q}_p \rightarrow \mathbb{Z}[1/p]$ such that
\begin{enumerate}[label=(\roman*)]
    \item\label{item-i} $|\alpha - s(\alpha)|_p < 1$ for every $\alpha \in \mathbb{Q}_p$;
\item\label{item-iii} $s(0)=0$; \
    \item\label{item-iv} if $|\alpha - \beta|_p < 1$, then $s(\alpha) = s(\beta)$.
\end{enumerate}
\end{definition}

Given $\alpha \in \QQ_p$ and a floor function $s$, the continued fraction expansion $[a_0,a_1,\ldots]_s$ of $\alpha$ is obtained via the algorithm: 
\[
\begin{cases}
\gamma_0 &=\alpha, \\
a_n &= s(\gamma_n), \\
\gamma_{n+1} &=\frac{1}{\gamma_n - a_n} \ \ \mbox{if } \gamma_n-a_n\neq 0.
\end{cases}
\]
and it stops if $\gamma_n=a_n$. Notice that, by definition, $|a_i|_p>1$ for every $i\geq 1$.

It is proved in \cite[Section 2, Theorem 2]{Browkin1978} that, for any choice of $s$, the map sending $\alpha$ to the partial quotients of its continued fraction expansion is a bijection between $\QQ_p$ and the (possibly infinite) sequences $a_0,a_1,a_2,\dots$ with $a_i \in \mathrm{Im}(s)$ and $|a_i|_p>1$ for $i \geq 1$. Since $a_0=0$ if and only if $|\alpha|_p<1$ (this follows directly from Definition~\ref{def:pFloor}), such map restricts to a bijection between $p \ZZ_p$ and the (possibly infinite) sequences $0,a_1,a_2,\dots$ with  $a_i \in \mathrm{Im}(s)$ and $|a_i|_p>1$ for $i \geq 1$.

\subsection{Some useful estimates on the continuants}\label{sect:continuants}
Given a floor function $s$ and a continued fraction 
 \[[a_0,a_1,a_2,\ldots ]_s\]
 we define, analogously to the real case, the classic sequences $(A_n)_{n\geq-1}$, $(B_n)_{n\geq -1}$ of \emph{continuants} given by the recurrences~\eqref{def-continuants}.
We have
\begin{equation}\label{convergent-shape} \frac{A_n}{B_n}=[a_0, \ldots, a_n]_s. \end{equation}
Notice that $A_n, B_n$ only depend on $a_0,\ldots, a_n$. 

Using matrices, we can write
$$
 \mathcal{A}_n =\begin{pmatrix} a_n &1 \\
   1 &0 \end{pmatrix},\quad\quad\quad
   \mathcal{B}_n =\begin{pmatrix} {A_{n}} &{A_{n-1}} \\
  {B_{n}} &{B}_{n-1} \end{pmatrix}, 
$$
for $n \in \mathbb N$. Moreover, if we define
\[\mathcal{B}_{-1} =\begin{pmatrix} 1\ &0 \\
  0\ & 1 \end{pmatrix},\]
 we have
\begin{equation}\mathcal{B}_n=\mathcal{B}_{n-1}\mathcal{A}_n
 =\mathcal{A}_0\mathcal{A}_1\ldots \mathcal{A}_n.
\label{eq:matriciale}
\end{equation}
This easily implies by induction the useful relation
\begin{equation} \label{eq:ABrel}
A_n B_{n-1}- B_n A_{n-1}=(-1)^{n+1}  \quad \mbox{for every } n\ge 0.
\end{equation}
\noindent Moreover, for every $k \ge 1$ we have 
\begin{equation} \label{eq:alpha}
    \alpha = [a_0, \ldots, a_{k-1}, \gamma]_s= \cfrac{\gamma A_{k-1}+A_{k-2}}{\gamma B_{k-1}+B_{k-2}}.
\end{equation}

We will also use the following properties.

\begin{lemma}\label{lem-An-Bn} Let $s$ be a $p$-adic floor function, $\alpha=[0,a_1,a_2,\ldots]_s$ and let $(A_n)_n, (B_n)_n$ be the sequences of continuants of $\alpha$.  Then:
\begin{enumerate}[label=(\roman*)]
\item\label{A-n-B_n-prod-aut} for every $n\geq 2$,  
$|A_n|_p=\prod_{i=2}^n|a_i|_p,\quad|B_n|_p=\prod_{i=1}^n|a_i|_p$.
\item\label{B_m-B_m+1}  for every $n\geq 0$, $|A_n|_p< |A_{n+1}|_p,\quad |B_n|_p<|B_{n+1}|_p,\quad
 |A_n|_p\leq |B_n|_p$.
 \end{enumerate}
 Suppose moreover that there exists a real number $M>0$ such that, for every $i$, $|a_i|_{\infty}\leq M$. Then 
 \begin{equation} \label{eq:bound_An_B_n_archimedean}
 \max(|A_n|_{\infty}, |B_n|_{\infty})\leq (M+1)^{n}.
 \end{equation}
\end{lemma}
\begin{proof}
Notice that
\[|A_0|_p=0, \, |A_1|_p=1, \, |A_2|_p=|a_2|_p, \,|A_3|_p=\max(|a_3|_p|a_2|_p,1)=|a_3|_p|a_2|_p\]  
and \[|B_0|_p=1, \, |B_1|_p=|a_1|_p, \, |B_2|_p=|a_1|_p|a_2|_p, \,|B_3|_p=|a_1|_p|a_2|_p|a_3|_p.\] 
Notice that, as $|a_i|_p> 1$, one can show by induction that, for every $n\geq 2$, $|A_{n-2}|_p<|a_{n}A_{n-1}|_p$ and similarly $|B_{n-2}|_p<|a_{n}B_{n-1}|_p$. This, together with the ultrametric inequality, proves
\ref{A-n-B_n-prod-aut}, which easily implies \ref{B_m-B_m+1}.
Finally, \eqref{eq:bound_An_B_n_archimedean} also easily follows by induction.
\end{proof}

The proof of the following useful lemma is obtained combining  \cite[Theorem 1]{Browkin1978} and \cite[Theorem 2]{Browkin1978}:
\begin{lemma}\label{lem:approx}  Let $s$ be a $p$-adic floor function, $\alpha=[a_0,a_1,a_2,\ldots]_s$ and let $(A_n)_n, (B_n)_n$ be the sequences of continuants of $\alpha$. Then for $n\geq 0$ 
\[
\left| B_n \alpha-  {A_n}\right |_p = \prod_{j=1}^{n+1} \frac{1}{|a_j|_p} =\frac{1}{ |B_{n+1}|_p}\leq \frac 1 {p^{n+1}}. 
\] In particular  \begin{equation}\label{eq-alpha-parti-quot}
        \left|\alpha- \frac{A_{n}}{B_{n}}\right|_p =  \frac{1}{\left|B_{n}B_{n+1}\right|_p}\leq \frac{1}{|B_n|_p^2}.\end{equation}
\end{lemma}

Notice that, by \ref{B_m-B_m+1} of Lemma \ref{lem-An-Bn}, for every $n\ge 0$, $|B_n|_p < |B_{n+1}|_p$, we have that the sequence  
$|\alpha -\frac{A_n}{B_n}|_p \rightarrow 0$
as $n$ tends to $\infty$; then, $\frac{A_n}{B_n}$ converges $p$-adically to $\alpha$.

The last lemma shows that, if the $n^{th}$ convergents of $\alpha$ and $\beta$ are the same, then the two numbers are $p$-adically closed. We have then the following result. 

\begin{lemma} \label{lemma2-nostro}Let $s$ be a $p$-adic floor function, $\alpha=[a_0,a_1,\ldots,a_n,a_{n+1}\ldots]_s$. Suppose that $\beta=[a_0,a_1,\ldots,a_n,b_{n+1},\ldots]_s$. 
    Then $\alpha$ and $\beta$ have the same first $n+1$ continuants $A_0,\ldots, A_n, B_0,\ldots,B_n$ and 
    $$ |\alpha-\beta|_p < \frac{1}{|B_n|_p^2}. $$
    
\end{lemma}

\begin{proof} 
By \eqref{convergent-shape} and \eqref{eq-alpha-parti-quot} we have
$$ \left | \alpha - \frac{A_n}{B_n} \right |_p< \frac{1}{|B_n|_p^2} \quad \mbox{and} \quad \left | \beta - \frac{A_n}{B_n} \right |_p< \frac{1}{|B_n|_p^2}. $$
By  adding and subtracting $\frac{A_n}{B_n}$ and applying the ultrametric inequality we obtain the conclusion.
\end{proof}

\subsubsection{Estimates for  Ruban's and Browkin's floor functions} \label{sub:Browkin_Ruban_floor}
The two main definitions of $p$-adic continued fraction algorithms are due to  Ruban \cite{Ruban1970} and Browkin \cite{Browkin1978}: in these cases the $p$-adic floor functions are given by 
$$ s(\alpha)=\sum_{n=k}^0 x_n p^n \in \QQ \ \ \mbox{where} \ \ \alpha=\sum_{n=k}^{+\infty} x_n p^n \in \QQ_p, $$
where $k\leq 0$ is an integer, the $x_n$ are representatives modulo $p$ in the interval $[0, p-1]$ for Ruban's and in the interval $[-(p-1)/2, (p-1)/2]$ for Browkin's  definition. Notice that, for these two floor functions, the partial quotients are bounded in absolute value by a constant $M$, which is $p$ in Ruban's case and $p/2$ in Browkin's. By \eqref{eq:bound_An_B_n_archimedean}, this implies that $  \max(|A_n|_{\infty}, |B_n|_{\infty})\leq (M+1)^{n}. $

\subsection{Periodic and palindromic continued fractions}
The following result is well known (see e.g. \cites[Th.\,III.1, p.\,40]{rockett1992continued}[Remark 6.1]{romeo2024continued}). Since it is widely used in our proofs, we reproduce here the proof for clarity in our case where $a_0=0$. We recall that, given a finite word $\omega$, we denote by $\overline{\omega}$ the infinite purely periodic word obtained repeating $\omega$. 
\begin{lemma}\label{lemma-polin-preperiodic}
    Let $s$ be a $p$-adic floor function and let \[\alpha=[0,a_1,\ldots,a_{w},\overline{a_{w+1},\ldots , a_\ell}]_s.\] Let $(A_n)_{n\geq -1},(B_n)_{n\geq -1}$ be the sequences of continuants of $\alpha$. Then $\alpha$ is a root of the polynomial \begin{equation}P(X)=a X^2-b X+c\end{equation}
    where \begin{equation}\label{coeff-poly}\begin{split} a& =B_{w -1}B_{\ell}- B_{w}B_{\ell-1}\\
b& =B_{w -1}A_{\ell}-B_{w}A_{\ell-1}+A_{w -1}B_{\ell}-A_{w} B_{\ell-1}\\
c &=A_{w-1}A_{\ell}-A_{w} A_{\ell-1}. \end{split}\end{equation}
    \end{lemma}
\begin{proof}
Let $\beta=[\overline{a_{w+1},\ldots , a_\ell}]_s$. Then
\[\alpha=\frac{A_{w}\beta+A_{w-1}}{B_{w}\beta+B_{w-1}}\]
but also 
\[\alpha=\frac{A_{\ell}\beta+A_{\ell-1}}{B_{\ell}\beta+B_{\ell-1}}.\]
Therefore $\alpha$ is a fixed point of the linear fractional transformation defined by the matrix
\[\begin{pmatrix} A_\ell & A_{\ell-1}\\ B_\ell & B_{\ell-1}\end{pmatrix}\begin{pmatrix} A_w & A_{w-1}\\ B_w & B_{w-1}\end{pmatrix}^{-1} =\pm \begin{pmatrix} A_\ell B_{w-1}-A_{\ell-1}B_w & -A_\ell A_{w-1}+A_{\ell-1}A_w\\
B_\ell B_{w-1}-B_{\ell-1}B_w & -B_\ell A_{w-1}+B_{\ell-1}A_w\end{pmatrix}. \]
The claim follows from \eqref{coeff-poly}.
\end{proof}

The following result establishes a relation between the continuants of a continued fraction and its palindrome.
\begin{lemma}
\label{lem:palindromeCF} Let $s$ be a $p$-adic floor function, let $\alpha= [a_0,a_1,\dots, a_n]_s$ and let $(A_n)_n, (B_n)_n$ be the sequences of continuants of $\alpha$. If $|a_0|_p>1$, then
     \begin{equation*}
        \frac{B_{n}}{B_{n-1}}=[a_n,\dots,a_0]_s.
    \end{equation*}
If $a_0=0$, we have
    \begin{equation*}
        \frac{B_{n-1}}{B_n}=[0,a_n,\dots,a_1]_s.
    \end{equation*}
\end{lemma}
\begin{proof}
    This can be easily checked by induction using the definition of $B_n$ from \eqref{def-continuants}.
\end{proof}

\section{Our main results and some applications}\label{sec:main-res-examples}
In this section we give an explicit version of our main result Theorem \ref{main-thm-intro}. We state separately  the part dealing with property $\spadesuit$ (see Theorem \ref{thm:pAdicBugeaud}) and the part dealing with property $\clubsuit$ (see Theorem \ref{main-thm-fiori}).

\subsection{An explicit version of Theorem \ref{main-thm-intro}}
\begin{theorem}
\label{thm:pAdicBugeaud}
     Let $p$ be an odd prime and let $s$ be a $p$-adic floor function.
    Let $\alpha=[0,a_1,a_2,\dots]_s$, and let $A_n, B_n$ be its $n$-th continuants.  
    
    Suppose that the word $a_1 a_2 a_3 \cdots$ satisfies condition $\spadesuit$ with constant $c$ and that there exist reals $C_{\infty}, C_p>1$ such that, for every $n\geq 1$:
    \begin{enumerate}[label=(\roman*)]
        \item\label{bound-a_i-arch}  $\max(|A_n|_{\infty}, |B_n|_{\infty})\leq C_{\infty}^n$;
        \item\label{B_n^1/n-is-bounded} $|B_n|_p\leq C_p^n$;
        \item \label{Ass}
$|a_n|_p \geq p^k$ where $k=\left\lfloor \max\left( 3\frac {\log C_\infty}{\log p}, 2(3c+1)\frac {\log C_\infty}{\log p} \right)\right\rfloor + 1$.
    \end{enumerate}
    Then $\alpha$ is either rational, quadratic, or transcendental.
\end{theorem}

\begin{remark}
\label{rem:rationalSpades}
Unlike the archimedean case, it can actually occur that $\alpha$ is rational in our statement. 

For example, let $s$ be the floor function defined by Ruban (see Section \ref{sub:Browkin_Ruban_floor}).
Then, as remarked in~\cite[§3]{Laohakosol1985},  one has \[-p=\left[0,\overline{\frac{p-1}{p}+(p-1)}\right]_s.\] 
So, the associated word satisfies condition $\spadesuit$ with constant $c=0$ and the conditions of Theorem \ref{thm:pAdicBugeaud} are satisfied (with $C_{\infty}=2$, $C_p=p$ and $k=1$).
\end{remark}

\begin{theorem}\label{main-thm-fiori}
     Let $p$ be an odd prime. Let $s$ be a $p$-adic floor function.
    Let $\alpha=[0,a_1,a_2,\dots]_s$, and let $A_n, B_n$ be its $n$-th continuants. 
    
    Suppose that the word $a_1 a_2 a_3 \cdots$ satisfies condition $\clubsuit$ with constant $c$ and that there exist reals $C_{\infty}, C_p>1$ such that, for every $n\geq 1$:
    \begin{enumerate}[label=(\roman*)]
    \item\label{archim-A_n-e-B_n^1/n-is-bounded-clubs}$\max(|A_n|_{\infty}, |B_n|_{\infty})\leq C_{\infty}^n$;
        \item\label{B_n^1/n-is-bounded-clubs}$|B_n|_p\leq C_p^n$;
\item \label{Ass-clubs} for every $n$,
$|a_n|_p \geq p^k$ where

\[k=\begin{cases}
 \left \lfloor  \frac{\log C_\infty}{\log p}\right \rfloor +1 \quad &\text{if $c= 0$,}\\[2ex]
   \left \lfloor (4+6c) \frac{\log C_\infty}{\log p}\right \rfloor +1 \quad &\text{if $c\neq 0$.}
\end{cases}
\]

%$k=\max\{ 3\frac {\log T}{\log p}, 2(3c+1)\frac {\log T}{\log p} \}$ and $T=(T'+\sqrt{{T'}^2+4})/2$.
    \end{enumerate}
    Then $\alpha$ is either rational, quadratic, or transcendental.
\end{theorem}

\begin{remark} \label{rem:finitealphabet}
Notice that:
\begin{itemize}
\item    For both Theorems~\ref{thm:pAdicBugeaud} and \ref{main-thm-fiori}, condition~\ref{B_n^1/n-is-bounded} implies the same bound also on $|A_n|_p$, since $|A_n|_p \le |B_n|_p$ by Lemma~\ref{lem-An-Bn}. However, this does not generally hold for the Archimedean absolute value, so we need to impose both bounds in condition~\ref{bound-a_i-arch}.
\item Moreover, if there exists $T>0$ such that $|a_i|_\infty\leq T$ for every $i$, as it happens in most definitions of $p$-adic continued fractions, condition~\ref{bound-a_i-arch} from Theorems \ref{thm:pAdicBugeaud} and \ref{main-thm-fiori} always holds true. Indeed, by Lemma~\ref{lem-An-Bn}, one can take $C_\infty= T+1.$
\end{itemize}
\end{remark}

\subsection{Some corollaries of our results}

\subsubsection{Ruban's and Browkin's floor functions: proof of Corollary \ref{cor:browRub}}
\label{subsec:applicationBrRub}
Since the word $a_1 a_2 a_3\cdots$ starts with arbitrarily long repetitions (resp.\ palindromes), it satisfies condition~$\spadesuit$ (resp.\ $\clubsuit$) with $c=0$. It is then enough to check that the hypotheses of Theorem~\ref{thm:pAdicBugeaud} (resp.\ Theorem~\ref{main-thm-fiori}) are satisfied. Note that infinite continued fraction expansions are always irrational in Browkin's case~\cite[Thm.\,3]{Browkin1978}, so the irrationality of $\alpha$ needs to be assumed only in Ruban's case.

 %By Remark~\ref{rem:finitealphabet},
 As observed in Section~\ref{sub:Browkin_Ruban_floor}, hypothesis~\ref{bound-a_i-arch} of Theorem~\ref{thm:pAdicBugeaud} (resp.\ Theorem~\ref{main-thm-fiori}) is automatically satisfied by setting $C_\infty=p+1$ when $s$ is Ruban's floor function, and $C_\infty=p/2 +1$ when $s$ is Browkin's floor function, applying Lemma \ref{lem-An-Bn}. Hypothesis~\ref{B_n^1/n-is-bounded} of Theorem~\ref{thm:pAdicBugeaud} (resp.\ Theorem~\ref{main-thm-fiori}) holds by assumption. Finally, hypothesis~\ref{Ass-clubs} of Theorem~\ref{thm:pAdicBugeaud} (resp.\ Theorem~\ref{main-thm-fiori}) can be easily checked, case by case, by setting $c=0$ and using the respective values of $C_\infty$. This concludes the proof of Corollary \ref{cor:browRub}.

\subsubsection{Words on a finite alphabet}\label{applfinitealphabet}

For words on a finite alphabet, hypotheses~\ref{bound-a_i-arch} and \ref{B_n^1/n-is-bounded} of Theorems~\ref{thm:pAdicBugeaud} and~\ref{main-thm-fiori} are always fulfilled. Therefore, our results can be specialized as follows.
\begin{corollary}
\label{cor:fin-alph}
   Let $p$ be an odd prime. Let $s$ be a $p$-adic floor function, and $\alpha=[0,a_1,a_2,\dots]_s$.   Suppose that the word $\omega = a_1 a_2 a_3 \cdots$ consists of letters from a finite alphabet and set $T=\max(|a_n|_\infty)_n$.
   Moreover, suppose that $\omega$ satisfies one of the following:
   \begin{enumerate}[label=(\arabic*)]
       \item \label{itm:finspade} condition $\spadesuit$ with constant $c$ and
        \begin{equation}
     \label{eqn:AssFiniteAlph1}
       (T+1)^{\max(3, 6c+2)}<p,
 \end{equation}
 \item \label{itm:finclub} condition $\clubsuit$ with constant $c$ and
 \begin{equation}
     \label{eqn:AssFiniteAlph2}
       \begin{cases}
           T<p-1 \quad&\text{if $c=0$,}\\
           (T+1)^{4+6c}<p\quad & \text{otherwise.} 
       \end{cases}
 \end{equation}
   \end{enumerate}

    Then $\alpha$ is either rational, quadratic, or transcendental.
\end{corollary}
\begin{proof}
It is enough to check that the conditions~\ref{bound-a_i-arch}, \ref{B_n^1/n-is-bounded}, and~\ref{Ass} in Theorem~\ref{thm:pAdicBugeaud} (resp.\, Theorem~\ref{main-thm-fiori}) are satisfied. By Remark~\ref{rem:finitealphabet}, ~\ref{bound-a_i-arch} is satisfied taking $C_{\infty}=T+1$. From part~\ref{A-n-B_n-prod-aut} of Lemma~\ref{lem-An-Bn}, one sees that \ref{B_n^1/n-is-bounded} holds for $C_p=\max(|a_i|_p)_i$. Finally,~\eqref{eqn:AssFiniteAlph1} (resp.\ \eqref{eqn:AssFiniteAlph2}) implies that $k= 1$ in ~\ref{Ass}, which is then automatically satisfied since $a_i\in \mathrm{Im}(s)\setminus\{0\}$.
\end{proof}

The following result specializes Corollary~\ref{cor:fin-alph} to low complexity words.
\begin{corollary}\label{corollario-automatiche}
    Let $p$ be an odd prime. Let $a,b\in\ZZ[1/p]$ such that: 
    \begin{enumerate}[label=(\roman*)]
        \item\label{itm:TM1} $|a-b|_p\geq 1$;
        \item\label{itm:TM2} $\min\left(|a|_p, |b|_p\right) \geq p$.
    \end{enumerate}
    Then, there exists a $p$-adic floor function $s$ such that $\{a,b\}\subset \mathrm{Im}(s)$. 

Let $s$ be such a floor function. Suppose moreover that 
    $\omega=a_1a_2a_3 \cdots$ is a word on the alphabet $\{a,b\}$ whose complexity function satisfies $p_{\omega}(n)\leq nC$ for all $n\geq 1$ where $C\geq 0$ is a real number such that:
    \begin{enumerate}[label=(\roman*)]
        \setcounter{enumi}{2}
        \item \label{itm:sublist} $\left(\max\left(|a|_\infty, |b|_\infty\right)+1\right)^{18C+8}<p$.
    \end{enumerate}
    Then $[0,a_1,a_2,\dots]_s$ is either rational, quadratic, or transcendental.
\end{corollary}
\begin{proof}
Notice that, in general, the construction of a $p$-adic floor function $s$ amounts to choosing a representative in $\ZZ[1/p]$ from any coset of $\mathbb{Q}_p$ modulo $p\mathbb{Z}_p$. Then the image of $s$ on the given class is the chosen element. Conditions \ref{itm:TM1} and \ref{itm:TM2} ensure that the cosets $a+p\ZZ_p$ and $b+p\ZZ_p$ are distinct and non zero. Hence one can choose any floor function $s$ mapping such cosets to $a$ and $b$ respectively (notice that there are infinitely many possible choices). 

We fix such an $s$. If $\omega$ is ultimately periodic, then, by classical results on $p$-adic continued fractions (see Lemma~\ref{lemma-polin-preperiodic}), the number $[0,a_1,a_2,\dots]_s$ is either rational or quadratic. Otherwise, by Theorem~\ref{bugeaud-complexity-picche}, $\omega$ satisfies condition $\spadesuit$ with constant $c=3C+1$. Therefore the hypotheses of Corollary~\ref{cor:fin-alph} -- in particular~\ref{itm:finspade}, which in this case is equivalent to~\ref{itm:sublist} -- are satisfied, and the conclusion follows immediately.
\end{proof}
Finally, Corollary~\ref{cor:largep} is a special case of the above result. %Corollary~\ref{cor:fin-alph}.
\begin{proof}[Proof of Corollary~\ref{cor:largep}]
    For $p$ large enough, $n,m,n-m$ are coprime to $p$. This is equivalent to saying that $n/p$ and $m/p$ satisfy hypotheses~\ref{itm:TM1} and~\ref{itm:TM2} of Corollary~\ref{corollario-automatiche}. Since also~\ref{itm:sublist} must be satisfied up to taking a larger $p$, the conclusion follows directly from Corollary~\ref{corollario-automatiche}.
\end{proof}

\begin{remark} Notice that condition \ref{itm:sublist} of Corollary~\ref{corollario-automatiche} can be made explicit, or even improved, for some  families of words (see Section \ref{rem-appendix} of Appendix \ref{app-A}). 
\end{remark}

\section{A key tool: the \texorpdfstring{$p$}{p}-adic Subspace Theorem}\label{sec:p-adic-subspace}

The proof of our main results requires a slightly more general version of the subspace theorem than the one by Schmidt used in \cite{Bugeaud2010AutomaticCF}, in order to include $p$-adic  places. 
We recall the classical definition of Weil height (see for instance \cite[\S 1.5.6]{BombieriGubler2006}). For our purposes, we restrict our attention to the case of an affine point in $ \ZZ[1/p]^N$.
\begin{definition}
\label{def:height}
    Let $\mathbf{z}$ be a vector $(z_{1},z_{2},\dots,z_{N})\in \ZZ[1/p]^N$. The \emph{(multiplicative) Weil height} of $\mathbf{z}$ is
\begin{equation*}
H(\mathbf{z})=\max(1,\max_i(|z_i|_\infty)) \cdot  \max(1,\max_i(|z_i|_p)).
\end{equation*}  
\end{definition}
We now state the generalisation of Schmidt's subspace theorem, due to Schlickewei \cite{Schlick1,Schlick2}, which can be obtained, for instance, as easy consequence of \cite[Theorem 2.3']{Bilu}. We refer to \cite{Bilu} for a comprehensive overview on the subject.
\begin{theorem}
\label{thm:subspace-appl}
Let $N\geq 1$ be an integer, let $S$ be the set containing the archimedean and the $p$-adic place of $\mathbb{Q}$.  For each $v \in S$, let $\{L_{1,v},\dots,L_{N,v}\}$ be a set of linearly independent  linear forms in the variables $X_1, \dots, X_N$ with algebraic  coefficients in $\mathbb{Q}_v$. 
 Let $(\mathbf{z}_n)_{n\geq 1}$ be a sequence of vectors in $\ZZ[1/p]^N$ and set $\mathbf{z}_n=(z_{n,1},z_{n,2},\dots,z_{n,N}) 
 $. Assume that there exist strictly positive reals $C_1, C_2, A, B$ with $A<1\leq B$ and an increasing sequence of positive integers $(u_n)_{n\geq 1}$ such that 
 \begin{align*} 
  \prod_{v \in S} \prod_{i=1}^N |L_{i,v}(\mathbf{z}_n)|_v &\leq C_1\cdot  A^{u_n} \\
  H(\mathbf{z}_n) & \leq C_2 \cdot  B^{u_n}.
  \end{align*}
 Then there exist $x_1,\ldots, x_N\in\ZZ$, not all zero, and an infinite subset $\mathcal{N}\subseteq \NN $ such that, for every $j\in \mathcal{N}$, one has
 \[x_1z_{j,1}+x_2z_{j,2}+\cdots + x_Nz_{j,N}=0.\]
\end{theorem}

\begin{proof}Up to normalizing the coefficients of the linear forms and increasing the constant $B$, we can assume that  \[ 
  \prod_{v \in S} \prod_{i=1}^N |L_{i,v}(\mathbf{z}_n)|_v\leq   A^{u_n} \hbox{ and }
  H(\mathbf{z}_n) \leq B^{u_n}.\]
Let $\epsilon>0$ be such that $A<1/B^\epsilon$. Then,
$$\prod_{v \in S} \prod_{i=1}^N |L_{i,v}(\mathbf{z}_n)|_v \leq  A^{u_n} < \frac 1 {B^{u_n\epsilon}}\leq \frac 1 {H(\mathbf{z}_n)^\epsilon}.$$
Then the result follows from \cite[Theorem 2.3']{Bilu}, by noticing that $H(\mathbf{z}_n)$ is larger or equal than the height of $\mathbf{z}_n$ as defined in \cite[p.5, formula (3)]{Bilu}. 
\end{proof}

In what follows, given two real sequences $(t_n)_{n\geq 1}$ and $(s_n)_{n\geq 1}$, we will use the notation $t_n \ll s_n$ to indicate that there exists a positive constant $C$ such that $t_n\leq C s_n$ for every $n \in \NN$.

\section{Proof of Theorem \ref{thm:pAdicBugeaud}}\label{sec:proof-picche}

Our proof strategy closely follows the one of~\cite[Thm.\,3.1]{Bugeaud2010AutomaticCF}.

Before proving the result we fix some notation. 
By hypothesis the sequence $(a_n)_{n\geq 1}$ satisfies condition $\spadesuit$ with constant $c$. Hence, there exist three sequences of finite words $(W_n)_n,(U_n)_n,(V_n)_n$ satisfying conditions \ref{it1}, \ref{it2} and \ref{it3} of Definition~\ref{def-spade-club-with-constant}. 

Setting $w_n=|W_n|$, $u_n=|U_n|$, and $v_n=|V_n|$, we have that, for every $n\geq 0$, $u_{n+1}\geq u_n$ and \begin{equation}\label{max-quot-w_n-u_n-v_n}
    \max(v_n/u_n,w_n/u_n)\leq c.
\end{equation}

We let also 
\[\ell_n=w_n+u_n+v_n.\]

We assume by contradiction that $\alpha$ is algebraic of degree at least $3$ over $\mathbb{Q}$.

\subsection{A preliminary lemma}
The proof of Theorem \ref{thm:pAdicBugeaud} is based on several applications of the $p$-adic Subspace Theorem (see Theorem \ref{thm:subspace-appl}), the first one of which is used to prove the following lemma. 
\begin{lemma}\label{picche-appl-subspace-1}
Let $A_n$ and $B_n$ as in Theorem \ref{thm:pAdicBugeaud} and $\ell_n, w_n$ as above. Then, there exist a non zero quadruple $(x_1,x_2,x_3,x_4)\in\ZZ^4$ and an infinite set $\mathcal{N}_1\subseteq \mathbb{N}$ such that
\begin{equation}\label{eq:Bugeaud-3.5}\begin{split} & x_1(B_{w_n-1}B_{\ell_n}- B_{w_{n}}B_{\ell_n-1})+x_2(B_{w_n-1}A_{\ell_n}-B_{w_{n}}A_{\ell_n-1})+\\
& +x_3(A_{w_n-1}B_{\ell_n}-A_{w_{n}}B_{\ell_n-1})+x_4(A_{w_n-1}A_{\ell_n}-A_{w_{n}}A_{\ell_n-1})=0
\end{split}
\end{equation}
for any $n\in \mathcal{N}_1$.
\end{lemma}

\begin{proof}
To simplify the notation, we set  
\[
\mathbf{z}_n=(z_{n,1},z_{n,2},z_{n,3},z_{n,4}),
\] 
where
\begin{equation}\label{eq:zeta}\begin{split} z_{n,1}& =B_{w_n -1}B_{\ell_n}- B_{w_{n}}B_{\ell_n-1}\\
z_{n,2} & =B_{w_n -1}A_{\ell_n}-B_{w_{n}}A_{\ell_n-1}\\
 z_{n,3} & = A_{w_n -1}B_{\ell_n}-A_{w_{n}}B_{\ell_n-1}\\
z_{n,4} &=A_{w_{n}-1}A_{\ell_n}-A_{w_{n}}A_{\ell_n-1}. \end{split}\end{equation}
We now want to apply the Subspace Theorem \ref{thm:subspace-appl} to the sequence $(\mathbf{z}_n)_n$ with carefully chosen linear forms.

To this aim, for $1\leq i\leq 4$ we choose the linear forms \[L_{i,\infty}(X_1,X_2,X_3,X_4)=X_i\] and 
\begin{align*}
L_{1,p}(X_1,X_2,X_3,X_4)&=\alpha^2 X_1-\alpha(X_2+X_3)+X_4, \\
L_{2,p}(X_1,X_2,X_3,X_4)&=\alpha X_1-X_2, \\
L_{3,p}(X_1,X_2,X_3,X_4)&=\alpha X_1-X_3, \\
L_{4,p}(X_1,X_2,X_3,X_4)&=X_1.
\end{align*}
We now want to estimate the quantities $|L_{i,\infty}(\mathbf{z}_n)|_{\infty}$ and $|L_{i,p}(\mathbf{z}_n)|_{p}$.

To this purpose, we set $\alpha_n=[W_n\overline{U_n V_n}]_s$. By Lemma \ref{lemma-polin-preperiodic}, $\alpha_n$ is a root of the following polynomial of degree 2 
\begin{equation*}
P_n(X)=z_{n,1} X^2-(z_{n,2}+z_{n,3})X+z_{n,4}.\end{equation*}

\noindent Moreover, by \eqref{eq-alpha-parti-quot},
    
   \[
        \left|\alpha- \frac{A_{n}}{B_{n}}\right|_p =  \frac{1}{\left|B_{n}B_{n+1}\right|_p}\]
        and     
  \begin{equation*}
  \left|\alpha_n- \frac{A_{\ell_n+u_n}}{B_{\ell_n+u_n}}\right|_p =  \frac{1}{\left|B_{\ell_n+u_n}B_{\ell_n+u_n+1}\right|_p}.
    \end{equation*}

\noindent Thus,
\begin{equation}
\label{eqn:alphaVSalphan}
    %|\alpha-\alpha_n|_p\leq \frac{1}{|B_{w_n+2u_n+v_n}B_{w_n+2u_n+v_n+1}|_p}.
    |\alpha-\alpha_n|_p\leq \frac{1}{|B_{\ell_n+u_n}B_{\ell_n+u_n+1}|_p}.
\end{equation}

Notice that, by Lemma \ref{lem-An-Bn} and  the fact that $|a_1|_p>1$, for every $m\geq 2$, we have
\begin{equation*}
%\label{A-n-B_n-prod}
|A_m|_p=\prod_{i=2}^m|a_i|_p< |A_{m+1}|_p,\quad |B_m|_p=\prod_{i=1}^m|a_i|_p < |B_{m+1}|_p, \quad |A_m|_p< |B_m|_p.\end{equation*}

Moreover, for every $n\geq 1$, we have 
\begin{equation}\label{max-p-adic-abs-val-z_{i,j}}
\max_{1\leq j\leq 4} |z_{n,j}|_p\leq |B_{w_{n}}B_{\ell_n}|_p.
\end{equation}

By \eqref{eq-alpha-parti-quot} and Lemma \ref{lem-An-Bn} \ref{B_m-B_m+1},  we have
\begin{align}\label{eq:3.2}
\begin{split}
|L_{2,p}(\mathbf{z}_n)|_p &= |(B_{w_{n}-1}B_{\ell_n}- B_{w_{n}}B_{\ell_n-1})\alpha- (B_{w_{n}-1}A_{\ell_n}-B_{w_{n}}A_{\ell_n-1})|_p\\
& \leq \max\left( |B_{w_{n}-1} |_p|B_{\ell_n}\alpha-A_{\ell_n}|_p,
|B_{w_{n}}|_p |B_{\ell_n-1}\alpha-A_{\ell_n-1}|_p\right)  \\
&\leq |B_{w_n}|_p|B_{\ell_n}|_p^{-1}
\end{split}
\end{align}
and
\begin{align}\label{eq:3.3}
\begin{split}
|L_{3,p}(\mathbf{z}_n)|_p &= |(B_{w_{n}-1}B_{\ell_n}- B_{w_{n}}B_{\ell_n-1})\alpha- (A_{w_{n}-1}B_{\ell_n}-A_{w_{n}}B_{\ell_n-1})|_p \\
& \leq \max\left( |B_{\ell_n} |_p|B_{w_{n}-1}\alpha-A_{w_{n}-1}|_p,
|B_{\ell_n-1} |_p|B_{w_{n}}\alpha-A_{w_{n}}|_p\right)  \\
&\leq |B_{\ell_n}|_p |B_{w_n}|_p^{-1}.
\end{split}
\end{align}

\noindent Also, by \eqref{eqn:alphaVSalphan}, \eqref{max-p-adic-abs-val-z_{i,j}}, \eqref{eq:3.2} and \eqref{eq:3.3}, we have
\begin{align}\label{eq:3.4}
\begin{split}
|L_{1,p}(\mathbf{z}_n)|_p  = &|P_n(\alpha)|_p 
= |P_n(\alpha)-P_n(\alpha_n)|_p=|z_{n,1}(\alpha^2-\alpha_n^2)-(z_{n,2}+z_{n,3})(\alpha-\alpha_n)|_p\\
=& |(\alpha-\alpha_n)|_p \cdot |(z_{n,1}\alpha-z_{n,2})+(z_{n,1}\alpha_n-z_{n,3})|_p=\\
=& |(\alpha-\alpha_n)|_p \cdot |L_{2,p}(\mathbf{z}_n)+L_{3,p}(\mathbf{z}_n)+z_{n,1}(\alpha_n-\alpha)|_p\\
\leq & |\alpha-\alpha_n|_p \max\left( |B_{w_n}|_p|B_{\ell_n} |_p^{-1}, |B_{w_n}|_p^{-1}|B_{\ell_n}|_p, |B_{w_n}|_p|B_{\ell_n}|_p|\alpha-\alpha_n|_p\right)   \\
\leq & |B_{w_n}|_p^{-1} |B_{\ell_n}|_p |B_{\ell_n+u_n}|_p^{-2}. 
\end{split}
\end{align}
Moreover, by \eqref{max-p-adic-abs-val-z_{i,j}}, we have 
\begin{equation}\label{eq:3.5}
|L_{4,p}(\mathbf{z}_n)|_p=|z_{n,1}|_p\leq |B_{w_{n}}B_{\ell_n}|_p.
\end{equation}

By the above formulas \eqref{eq:3.2}, \eqref{eq:3.3}, \eqref{eq:3.4}, \eqref{eq:3.5}, we have that 
\begin{equation*}\label{bound-linear-forms}
\prod_{j=1}^4 |L_{j,p}(\mathbf{z}_n)|_p \leq |B_{\ell_n}|^2_p |B_{\ell_n+u_n}|^{-2}_p.
 \end{equation*}
Since, by the assumption \ref{Ass} of Theorem \ref{thm:pAdicBugeaud},
 \begin{equation*}
     |B_{\ell_n+u_n}|_p= |B_{\ell_n}|_p\prod_{i=\ell_n+1}^{\ell_n+u_n} |a_i|_p\geq |B_{\ell_n}|_p \cdot p^{ku_n},
 \end{equation*}
 and, by \ref{bound-a_i-arch},
 \begin{equation}\label{bound-forme-var-arch}\max_{1\leq j\leq4}|{z}_{n,j}|_\infty\leq 2 C_{\infty}^{2w_n+u_n+v_n}\leq 2 C_{\infty}^{(3c+1)u_n},
 \end{equation}
 we find 
 \begin{equation}\label{bound-linear-forms-1}
\prod_{j=1}^4 |L_{j,p}(\mathbf{z}_n)|_p \prod_{j=1}^4 |z_{n,j}|_\infty 
\leq 16  \left (\frac {C_{\infty}^{4(3c+1)}} {p^{2k}}\right)^{u_n}.\end{equation}
Using  \eqref{max-p-adic-abs-val-z_{i,j}}, \eqref{bound-forme-var-arch}, and assumption~\ref{B_n^1/n-is-bounded} of Theorem~\ref{thm:pAdicBugeaud}, we also have
\begin{equation}\label{bound-altezza}
H(\mathbf{z}_n)\leq  {2 C_{\infty}^{(3c+1)u_n}}|B_{w_n}|_p|B_{\ell_n}|_p \leq 2\cdot  \left((C_{\infty} C_p)^{3c+1}\right)^{u_n}.\end{equation}
Then, by Theorem \ref{thm:subspace-appl}, the conclusion follows provided that the quantity $\frac {C_{\infty}^{4(3c+1)}} {p^{2k}}$ in~\eqref{bound-linear-forms-1} is $<1$, and this is in fact the case by hypothesis \ref{Ass} of Theorem \ref{thm:pAdicBugeaud}. 
\end{proof}

\subsection{Splitting the proof in two cases}
We are now ready to prove Theorem \ref{thm:pAdicBugeaud}. By Lemma~\ref{picche-appl-subspace-1}, we may consider only the indices in the subset $\mathcal{N}_1$ so that~\eqref{eq:Bugeaud-3.5} is satisfied. 

Here the proof is split in two cases, depending on whether the sequence $(w_n)_{n\in \mathcal{N}_1}$ has a constant subsequence or not. 

In the first case, using the continuants of $\alpha$, Lemma \ref{picche-appl-subspace-1} and the Subspace theorem, we construct an irrational number $\beta$ such that both $(1,\alpha,\beta)$ and $(1,\alpha,1/\beta)$ are linearly dependent over $\mathbb{Z}$. Combining these linear relations, we obtain a contradiction. 

In the second case, using Lemma \ref{picche-appl-subspace-1} and suitable auxiliary words, we show that $\alpha$ is a root of a certain polynomial of degree at most 2, which must therefore be trivial. The vanishing of its coefficients, gives relations among the continuants of $\alpha$ and these, together with the Subspace theorem, allow to conclude.  

\subsubsection*{\bf Case (1): $(w_n)_{n\in \mathcal{N}_1}$ has a constant subsequence}
Let $w$ be such that $w_n=w$ for infinitely many $n\in \mathcal{N}_1$.
By extracting an infinite subset of $\mathcal{N}_1$ and replacing $\alpha$ by \[[0,a_{ w+1},a_{ w+2},\ldots]_s,\] we may without loss of generality assume that $w_n=0$ and hence $\ell_n=u_n+v_n$ for every $n\in \mathcal{N}_1$.
Let now  \[\beta=\lim_{\substack{n\rightarrow \infty\\ n\in \mathcal{N}_1}}\frac {B_{\ell_n-1}}{B_{\ell_n}}.\]
We claim that:
\begin{enumerate}[label=(\alph*)]
\item\label{cond-beta-i} $\beta$ is irrational;
\item\label{cond-beta-ii} there is a non zero triple of integers $(y_1,y_2,y_3)$ such that \begin{equation} \label{eq:subspace2bis} y_1+y_2\beta+y_3\alpha=0;
\end{equation} 
\item\label{cond-beta-iii} there is a non zero triple of integers $(z_1,z_2,z_3)$ such that \begin{equation} \label{eq:subspace3bis} \frac{z_1} \beta + z_2 +z_3\alpha =0.\end{equation}
\end{enumerate}
Assuming these claims, from \eqref{eq:subspace2bis} and \eqref{eq:subspace3bis} we obtain
$$(z_3\alpha+z_2)(y_3\alpha+y_1)=y_2z_1.$$
Since from \ref{cond-beta-i} $\alpha$ and $\beta$ are both irrational, we get from \eqref{eq:subspace2bis} and \eqref{eq:subspace3bis} that $y_3z_3\not=0$. This shows that $\alpha$ is an algebraic number of degree at most two, which is a contradiction with our assumption that $\alpha$ is algebraic of degree at least 3. This concludes the proof in this case.

We now prove the above claims. Let $(x_1, x_2, x_3, x_4)\in \mathbb{Z}^4$ and $\mathcal{N}_1\subseteq \mathbb{N}$ be as in Lemma \ref{picche-appl-subspace-1}. 
Recalling that $B_{-1}=A_0=0$ and $B_0=A_{-1}=1$, we deduce from 
\eqref{eq:Bugeaud-3.5}
that
\begin{equation}\label{eq:Bugeaud-3.6} x_1B_{\ell_n-1}+x_2A_{\ell_n-1}-x_3B_{\ell_n}-x_4A_{\ell_n}=0
\end{equation}
for any $n\in \mathcal{N}_1$. Dividing 
\eqref{eq:Bugeaud-3.6} by $B_{\ell_n}$ and letting $n$ tend to infinity in $\mathcal{N}_1$ we get
\begin{equation}\label{eq:beta}\beta=\frac {x_3+x_4\alpha}{x_1+x_2\alpha}.\end{equation}
Observe that $x_1+x_2\alpha\neq 0$, since otherwise $\alpha$ would be rational.

\medskip
We now want to prove \ref{cond-beta-i} and \ref{cond-beta-ii} (notice that \ref{cond-beta-i} is not  automatically guaranteed by \eqref{eq:beta} since a priori one could have $x_3x_2-x_4x_1=0$).

This is again obtained by an application of the Subspace Theorem \ref{thm:subspace-appl} with $N=3$ and the linear forms, for $1\leq i\leq 3$ $$L_{i,\infty}'(X_1,X_2,X_3)=X_i$$
and 
$$L'_{1,p}(X_1,X_2,X_3)=\beta X_1-X_2, \quad L'_{2,p}(X_1,X_2,X_3)=\alpha X_1-X_3, \quad L'_{3,p}(X_1,X_2,X_3)=X_2,$$
evaluated at the points  $(B_{\ell_n},B_{\ell_n-1}, A_{\ell_n})$.

Setting $Q_n=\frac {A_n}{B_n}$, we have, by~\eqref{eq:Bugeaud-3.6},
\[\frac{B_{\ell_n-1}}{B_{\ell_n}}=\frac{x_3+x_4Q_{\ell_n}}{x_1+x_2 Q_{\ell_n-1}},\]
and therefore, by~\eqref{eq:beta},
\[
\left |\beta - \frac{B_{\ell_n-1}} {B_{\ell_n}} \right |_p 
= \left |\frac{x_1x_4(\alpha-Q_{\ell_n})-x_2x_3(\alpha-Q_{\ell_n-1})+x_2x_4\alpha(Q_{\ell_n-1}-Q_{\ell_n})}  {(x_1+x_2\alpha)(x_1+x_2Q_{\ell_n-1})} \right |_p.\] 
We notice that, since $Q_{\ell_n-1}$ converges $p$-adically to $\alpha$ and $x_1+x_2\alpha \neq 0$ since $\alpha$ is irrational, we have that, for $n$ sufficiently large, $|x_1+x_2Q_{\ell_n}|_p=|x_1+x_2\alpha|_p$. Hence, for $n \in \mathcal N_1$, we have
\begin{equation}\label{eq:betaVSB}
|B_{\ell_n}\beta- B_{\ell_{n}-1}|_p\leq \frac{C'}{|B_{\ell_n-1}|_p}
\end{equation}
for some real $C'>0$.

Furthermore, observe that, for any integer $n\in\mathcal{N}_1$, we have, by Lemma \ref{lem:approx} and~\eqref{eq:betaVSB}, 
\begin{equation}\label{eq:stimaalphabeta}\max\left (|B_{\ell_{n}}\alpha- A_{\ell_{n}}|_p , |B_{\ell_{n}-1}\alpha- A_{\ell_{n}-1}|_p ,|B_{\ell_n}\beta- B_{\ell_{n}-1}|_p\right )\ll \frac{1}{{|B_{\ell_n-1}|_p}}.   \end{equation} 
We obtain, by \eqref{eq:stimaalphabeta} and hypotheses~\ref{bound-a_i-arch} and~\ref{Ass} of Theorem~\ref{thm:pAdicBugeaud},
\begin{equation}
    \label{eqn:leftSchmidt}
    \prod_{1\leq j\leq 3} |L'_{j,p}(B_{\ell_n},B_{\ell_n-1}, A_{\ell_n})|_p |L'_{j,\infty}(B_{\ell_n},B_{\ell_n-1}, A_{\ell_n})|_{\infty} \ll  \frac {C_{\infty}^{3\ell_n}}{|B_{\ell_n-1}|_p}\ll \left (\frac {C_{\infty}^3} {p^k}\right)^{\ell_n} 
\end{equation}
and
\begin{equation}\label{stima-H-subspace2}H(B_{\ell_n},B_{\ell_n-1}, A_{\ell_n})\ll (C_{\infty} C_p)^{\ell_n}. \end{equation}
Notice that, up to passing to a subset of $\mathcal{N}_1$, we might assume that $\ell_n$ is increasing. By the definition of $k$ in Theorem \ref{thm:pAdicBugeaud} we have
$C_\infty^3/p^k<1$ in~\eqref{eqn:leftSchmidt}. Therefore, by Theorem \ref{thm:subspace-appl},  there exist a non zero integer triple $(y_1,y_2,y_3)$ and an infinite set of distinct positive integers $\mathcal{N}_2\subseteq \mathcal{N}_1$ such that
\begin{equation} \label{eq:subspace2} y_1B_{\ell_n}+y_2B_{\ell_{n}-1}+y_3A_{\ell_n}=0\end{equation}
for any $n\in\mathcal{N}_2$.  Dividing \eqref{eq:subspace2} by $B_{\ell_n}$ and letting $n$ tend to infinity along $\mathcal{N}_2$, we get that
\eqref{eq:subspace2bis} holds, so that claim~\ref{cond-beta-ii} is proven.

Equations~\eqref{eq:subspace2} and~\eqref{eq:subspace2bis} then prove claim \ref{cond-beta-i}: indeed, if $\beta$ was rational,~\eqref{eq:subspace2bis} would imply $y_3=0$ (recall that we are supposing $\alpha$ irrational). But then, by \eqref{eq:subspace2} we would have \[
B_{\ell_n-1}/B_{\ell_{n}}=-y_1/y_2,
\]  for all $n\in\mathcal{N}_2$ and, for each such index, by Lemma~\ref{lem:palindromeCF} we would have
\begin{equation*}
    \label{eq:BnBn+1}
    -\frac{y_2}{y_1}=\frac{B_{\ell_n-1}}{B_{{\ell_n}}}=[0,a_{\ell_n},\dots,a_1]_s,
\end{equation*} 
contradicting the uniqueness of $p$-adic continued fraction expansion of $-y_2/y_1$ attached to the floor function $s$.\footnote{Notably, this is the first time in the proof where we use the fact that the word associated to $\alpha$ arises from a given floor function.}

To prove \ref{cond-beta-iii}, we also apply the Subspace Theorem \ref{thm:subspace-appl} with $N=3$ and the linear forms
$$L''_{i,\infty}(X_1,X_2,X_3)=X_i$$
for $1\leq i\leq 3$ and
$$L''_{1,p}(X_1,X_2,X_3)=\beta X_1-X_2,\quad L''_{2,p}(X_1,X_2,X_3)=\alpha X_2-X_3,\quad L''_{3,p}(X_1,X_2,X_3)=X_2.$$
Evaluating them on the triple  $(B_{\ell_n},B_{\ell_n-1}, A_{\ell_{n-1}})$ with $n\in\mathcal{N}_2$, we infer from \eqref{eq:stimaalphabeta} that 
\[\prod_{1\leq j\leq 3} |L''_{j,p}(B_{\ell_n},B_{\ell_n-1}, A_{\ell_n-1})|_p|L''_{j,\infty}(B_{\ell_n},B_{\ell_n-1}, A_{\ell_n-1})|_{\infty}\ll \left(\frac {C_{\infty}^{3}}{p^k}\right)^{\ell_n}.\]
As before, estimating $H(B_{\ell_n},B_{\ell_n-1}, A_{\ell_n-1})$ as in \eqref{stima-H-subspace2}, and using the definition of $k$ in Theorem \ref{thm:pAdicBugeaud}, we can apply Theorem \ref{thm:subspace-appl} and deduce that there exist a non zero integer triple $(z_1,z_2,z_3)$ and an infinite subset $\mathcal{N}_3\subseteq \mathcal{N}_2$ such that
\begin{equation} \label{eq:subspace3} z_1B_{\ell_n}+z_2B_{\ell_n-1}+z_3A_{\ell_{n}-1}=0\end{equation}
for any $n\in\mathcal{N}_3$.  Dividing \eqref{eq:subspace3} by $B_{\ell_{n-1}}$ and letting $n$ tend to infinity along $\mathcal{N}_3$, we get \eqref{eq:subspace3bis}, concluding the proof of this case.

\subsubsection* {\bf Case (2): $(w_n)_{n\in \mathcal{N}_1}$ has no constant subsequences} As detailed in \cite[page 1013]{Bugeaud2010AutomaticCF}, up to suitably modifying the words $U_n$ and $W_n$, we can  assume that the last letter of $U_nV_n$ is different from that of $W_n$, that is  $a_{w_n}\not = a_{\ell_n}$ for any $n\in\mathcal{N}_1$. 

We now set 
\begin{equation}\label{eq:Qn} Q_n= \frac {B_{w_n-1}B_{\ell_n}}{B_{w_n}B_{\ell_n-1}}\end{equation} 
and 
\begin{equation}\label{def-R_m}R_m=\alpha-\frac {A_m}{B_m}.\end{equation}

\noindent Dividing \eqref{eq:Bugeaud-3.5} by $B_{w_n}B_{\ell_n-1}$, we obtain
\begin{equation*}\begin{split} & x_1(Q_n-1)+x_2(Q_n(\alpha-R_{\ell_n})-(\alpha-R_{\ell_{n}-1})) +x_3(Q_n(\alpha-R_{w_n-1})-(\alpha-R_{w_n})) +\\
& + x_4(Q_n(\alpha-R_{w_n-1})(\alpha-R_{\ell_n})-(\alpha-R_{w_n})(\alpha-R_{\ell_n-1}))=0.
\end{split}\end{equation*}
Letting also \[P(\alpha)=x_1+(x_2+x_3)\alpha+x_4\alpha^2,\] this yields
\begin{equation}\label{eq:3.14-Bugeaud}\begin{split} (Q_n-1)P(\alpha)=& 
x_2Q_nR_{\ell_n}-x_2R_{\ell_n-1}+x_3Q_nR_{w_n-1}-x_3R_{w_n}
- x_4Q_nR_{w_n-1}R_{\ell_n}+\\ &+x_4R_{w_n}R_{\ell_n-1}+\alpha x_4(Q_nR_{w_n-1}+Q_nR_{\ell_n}-R_{w_n}-R_{\ell_n-1}).
\end{split}
\end{equation}
We want to prove 
\begin{equation}\label{eq:claim-Bugeaud}
P(\alpha)=0.
\end{equation}

\noindent Suppose first that there exists $\epsilon>0$ and an infinite subset  $\mathcal{M}\subset \mathcal{N}_1$ such that $|Q_n-1|_p\geq \epsilon$ for all $n\in \mathcal{M}$.

Suppose further that  $|Q_n|_p\leq 1$ for infinitely many $n\in\mathcal{M}$.
Then, for $n$ tending to infinity along the elements of this subset, the right-hand side of \eqref{eq:3.14-Bugeaud} tends to 0 (indeed, by~\eqref{def-R_m}, $R_{m}$ tends to 0 for $m$ tending to infinity and $Q_m$ is bounded), which  implies~\eqref{eq:claim-Bugeaud}. 
If instead $|Q_n|_p>1$ for all but finitely many $n\in\mathcal{M}$, dividing both sides of formula \eqref{eq:3.14-Bugeaud} by $Q_n$ and letting $n$ tend to infinity, we obtain, as in the previous case, that \eqref{eq:claim-Bugeaud} holds. 

%\[\lim_{n\in \mathcal{N}_1}|Q_n-1|_p\neq 0.\] Up to swapping $Q_n$ and $Q_{n}^{-1}$, we can assume that there exists an infinite subset of indices $n\in \mathcal{N}_1$ such that $|Q_n-1|_p\geq \epsilon$ and $|Q_n|_p\leq 1$.

Suppose now that
\[\lim_{n\in \mathcal{N}_1}|Q_n-1|_p= 0.\]
Then, for $n$ sufficiently large, we have $|Q_n|_p=|1|_p=1$. 

Notice that, by Lemma \ref{lem:palindromeCF}, $Q_n$ is the quotient of the two continued fractions $\mathbf{a}_{\ell_n}=[a_{\ell_n}, a_{\ell_n-1},\dots,a_1]_s$ and $\mathbf{a}_{w_n}=[a_{w_n}, a_{w_n-1},\dots,a_1]_s$, hence must have $$|a_{\ell_n}|_p=|\mathbf{a}_{\ell_n}|_p=|\mathbf{a}_{w_n}|_p=|a_{w_n}|_p=p^{k_n},$$ for $n$ sufficiently large, with $k_n>0$. 

Let us now assume, by contradiction, that $P(\alpha)\neq 0$. Then, since $|R_\ell|_p \ll 1/|B_\ell B_{\ell+1}|_p$ for $\ell\geq 1$, from~\eqref{eq:3.14-Bugeaud}  we deduce
\begin{equation*}
    |Q_n - 1|_p \ll \frac{1}{|B_{w_n-1}B_{w_n}|_p},
\end{equation*}
which we can rephrase as 
 \[
   \left|\frac{\mathbf{a}_{\ell_n}}{\mathbf{a}_{w_n}}-1\right|_p\ll \frac{1}{|a_{w_n}|_p\cdot\prod_{i=1}^{w_n-1}|a_i|^2_p}.\]

Now, multiplying both the numerator and denominator of the left-hand side by $p^{k_n}$ and noticing that $|p^{k_n}\mathbf{a}_{w_n}|_p=1$ and $|a_i|_p\ge p$ for every $i\le 1$, we obtain, in particular, that 
\[
    \left|p^{k_n} \mathbf{a}_{\ell_n}-p^{k_n}\mathbf{a}_{w_n}\right|_p\ll \frac{1}{p^{2(w_n-1)+k_n}}.
    \]
    Therefore
    \[
    \left|[a_{\ell_n}, a_{\ell_n-1},\dots,a_1]_s-[a_{w_n}, a_{w_n-1},\dots,a_1]_s\right|_p\ll \frac{1}{p^{2(w_n-1)}}\leq \frac 1 p.
    \]
This implies that $\mathbf{a}_{\ell_n}\equiv\mathbf{a}_{w_n}\bmod p $, so that $a_{\ell_n}\equiv a_{w_n}\bmod p $. Therefore, by property~\ref{item-iv} in the definition of floor function, they are equal, contradicting our initial assumption that  $a_{w_n}\not = a_{\ell_n}$.

Therefore \eqref{eq:claim-Bugeaud} is true; since we are assuming that $\alpha$ is neither rational nor quadratic, we must have $x_1=x_2+x_3=x_4=0$. Then, equation \eqref{eq:Bugeaud-3.5} gives
\begin{equation*} B_{w_{n-1}}A_{\ell_n}-B_{w_{n}}A_{\ell_n-1}=A_{w_{n-1}}B_{\ell_n}-A_{w_{n}}B_{\ell_n-1}.
\end{equation*}
In particular, in \eqref{eq:zeta} we have $z_{n,2}=z_{n,3}$.
We can thus refine the argument in Lemma \ref{picche-appl-subspace-1}. Let $\alpha_n=[W_n,\overline{U_nV_n}]_s$. Then, by Lemma~\ref{lemma-polin-preperiodic}, $\alpha_n$ is a root of the polynomial 
\begin{equation*} 
P_n(X)=z_{n,1}X^2-2 z_{n,2}X+z_{n,4}.
\end{equation*}

\noindent Consider now the two sets of linearly independent linear forms
\begin{align*}
L'''_{1,p}(T_1,T_2,T_3) &= \alpha^2T_1-2\alpha T_2+T_3,\\
L'''_{2,p}(T_1,T_2,T_3) &= \alpha T_1-T_2,\\
L'''_{3,p}(T_1,T_2,T_3) &= T_1,\\
L'''_{1,\infty}(T_1,T_2,T_3) &= T_1,\\ 
L'''_{2,\infty}(T_1,T_2,T_3) &= T_2,\\
L'''_{3,\infty}(T_1,T_2,T_3) &= T_3;
\end{align*}
evaluating them on the triple $\mathbf{z}'_n=(z_{n,1},z_{n,2},z_{n,4})$ we find by  \eqref{eq:3.2}, \eqref{eq:3.4}, \eqref{eq:3.5}, and \eqref{bound-forme-var-arch} that
\begin{align*}\prod_{1\leq j\leq 3}|L'''_{j,p}(\mathbf{z}'_n)|_p\prod_{1\leq j\leq 3}|L'''_{j,\infty }(\mathbf{z}'_n)|_\infty &\ll |B_{w_n}|_p|B_{\ell_n}|_p|B_{\ell_n+u_n}|_p^{-2}C_\infty ^{3(\ell_n+w_n)}\\
&\leq \frac {C_\infty ^{3(3c+1)u_n}}{\prod_{j=\ell_n+1}^{\ell_n+u_n}|a_j|^2_p}\\
&\leq \left (\frac {C_\infty^{3(3c+1)}}{p^{2k}}\right)^{u_n}.
\end{align*}
Also, by \eqref{bound-altezza} we have
\[
    H(\mathbf{z}_n)\ll  (C_\infty C_p )^{(3c+1) u_n}.
\]

Again using the definition of $k$ in Theorem \ref{thm:pAdicBugeaud} and applying Theorem \ref{thm:subspace-appl} we obtain that there exists a non zero integer triple $(t_1,t_2,t_3)$ and an infinite set  $\mathcal{N}_4\subseteq \mathcal{N}_1$ such that
\begin{equation}\label{eq:3.17-Bugeaud} t_1 z_{n,1}+t_2 z_{n,2}+t_3z_{n,4}=0\end{equation}
for any $n\in\mathcal{N}_4$.
We proceed exactly as above. Divide \eqref{eq:3.17-Bugeaud} by $B_{w_n}B_{\ell_n-1}$ and set $Q_n$ as in \eqref{eq:Qn}. Then, we get
\begin{equation*}\label{eq:3.18-Bugeaud}
 t_1(Q_n-1)+t_2\left (Q_n\frac {A_{\ell_n}}  {B_{\ell_n}}- \frac {A_{\ell_{n}-1}}  {B_{\ell_{n}-1}}\right )+t_3\left (Q_n\frac {A_{w_{n}-1}}  {B_{w_{n}-1}}\frac {A_{\ell_{n}}}  {B_{\ell_{n}}}-\frac{A_{w_{n}}}  {B_{w_{n}}}\frac{A_{\ell_{n}-1}}  {B_{\ell_{n}-1}}\right )=0,
\end{equation*}
for any $n\in\mathcal{N}_4$. By arguing as after \eqref{eq:claim-Bugeaud} we derive that $t_1+t_2\alpha+t_3\alpha^2=0$, a contradiction since $\alpha$ is irrational and not quadratic.
This concludes the proof of Theorem \ref{thm:pAdicBugeaud}. \qed
\section{Proof of Theorem \ref{main-thm-fiori}}\label{sec:proof-fiori}
Suppose by contradiction that $\alpha$ were algebraic of degree at least $3$ over $\mathbb{Q}$.

By hypothesis the sequence $(a_n)_n$ satisfies condition $\clubsuit$ with constant $c$. Hence, there exist three sequences of finite words $(W_n)_n,(U_n)_n,(V_n)_n$ satisfying conditions \ref{it1}, \ref{it2} and \ref{it3} of Definition \ref{def-spade-club-with-constant}. 
Setting $w_n=|W_n|$, $u_n=|U_n|$, and $v_n=|V_n|$, for every $n\geq 0$,
\begin{equation*}
\max(v_n/u_n,w_n/u_n)\leq c.
\end{equation*}
  We let also 
\[\ell_n=w_n+u_n+v_n.\]

Here the proof splits into three cases.  

The first case occurs when the set $\{n \geq 1 \mid v_n=w_n = 0\}$ is  infinite: here, using the fact that the words $a_1 \cdots a_{2u_n}$ are palindromic for infinitely many $n$, an  application of the Subspace Theorem \ref{thm:subspace-appl} allows us to conclude. 

In the second case, we consider the situation where the set $\{n \geq 1 \mid v_n=w_n = 0\}$ is finite, but the set $\{n \geq 1 \mid w_n = 0\}$ is still infinite: in this case, on one hand we show that the continued fraction expansion of $\alpha$ has `many' partial quotients in common with that of certain non palindromic words constructed using certain continuants of $\alpha$. On the other hand, an application of the Subspace theorem and the uniqueness of the continued fraction expansion for a given floor function, imply that the words we considered are actually palindromic, leading to a contradiction.

Finally, the third case occurs when  $\{n \geq 1 \mid w_n = 0\}$ is finite.  Here, we use the words $U_n$, $V_n$ and $W_n$ together with the recurrence formulas defining the continuants to build two sequences of numbers whose continued fraction expansions have `many' partial quotients in common with that of $\alpha$. A first application of the Subspace theorem gives a relation between some of these words and a second application of the Subspace theorem allows us to conclude.

\subsubsection*{Case (1):  the set $\{n\geq 1\mid w_n=v_n=0\}$ is infinite.}
In this case (which occurs, in particular, when $c=0$) we have that, up to passing to an infinite subset of positive integers $n$, we may assume that $\ell_n=u_n$ and $U_n \widetilde{U_n}$ is a prefix of $\alpha$ for every $n$. Since the word $a_1\cdots a_{2u_n}=U_n \widetilde{U_n}$ is palindromic, we have that
 the matrix
\[\begin{pmatrix}
B_{2u_n}& B_{2u_n-1}\\
A_{2u_n}& A_{2u_n-1}
\end{pmatrix}=
\begin{pmatrix}
a_{1}& 1\\
1& 0
\end{pmatrix}\begin{pmatrix}
a_{2}& 1\\
1& 0
\end{pmatrix}\cdots \begin{pmatrix}
a_{2u_n}& 1\\
1& 0
\end{pmatrix}\]
is symmetric, hence, for every $n$ \begin{equation}\label{A_2un=B_2un-1}
    A_{2u_n}=B_{2u_n-1}.
\end{equation} 

Notice that by \eqref{eq-alpha-parti-quot} we have
\begin{equation}\label{form1-palind}
    \left|\alpha-\frac{A_{2u_n}}{B_{2u_n}}\right|_p\leq \frac{1}{|B_{2u_n}|_p^2}
\end{equation}
and also, by~\eqref{A_2un=B_2un-1},
\begin{equation}\label{form1bis-palind}
    \left|\alpha-\frac{A_{2u_n-1}}{A_{2u_n}}\right|_p=\left|\alpha-\frac{A_{2u_n-1}}{B_{2u_n-1}}\right|_p\leq \frac{1}{|B_{2u_n-1}|_p^2}.
\end{equation}

Moreover
\begin{align*}
\left|\alpha^2-\frac{A_{2u_n-1}}{B_{2u_n}}\right|_p & =  \left |\alpha^2-\frac{A_{2u_n-1}}{B_{2u_n-1}} \frac{A_{2u_n}}{B_{2u_n}} \right |_p\\
&= \left|\left (\alpha + \frac{A_{2u_n-1}}{B_{2u_n-1}}\right) \left (\alpha -\frac{A_{2u_n}}{B_{2u_n}} \right) +\alpha \left (\frac{A_{2u_n}}{B_{2u_n}}-\frac{A_{2u_n-1}}{B_{2u_n-1}}\right ) \right |_p\\
&\leq \max\left(|\alpha|_p\left |\alpha-\frac{A_{2u_n}}{B_{2u_n}}\right|_p, \frac{|\alpha|_p}{|B_{2u_n}B_{2u_n-1}|_p}\right) \end{align*}
 and since $|\alpha|_p<1$, we finally obtain
 \begin{equation}\label{form2-palind}
\left|\alpha^2-\frac{A_{2u_n-1}}{B_{2u_n}}\right|_p < \max\left(\frac 1 {|B_{2u_n}|_p^2}, \frac 1 {|B_{2u_n}B_{2u_n-1}|_p}\right) \leq \frac{1}{|B_{2u_n}B_{2u_n-1}|_p}.
\end{equation}

We apply Theorem \ref{thm:subspace-appl} with $N=3$ and the linear forms \begin{align*}L_{1,p}'(X_1,X_2,X_3)&=X_1\alpha-X_2,
\\ L_{2,p}'(X_1,X_2,X_3)&=X_1\alpha^2-X_3,\\
L_{3,p}'(X_1,X_2,X_3)&=X_2\alpha-X_3,\end{align*}
and $L_{i,\infty}'(X_1,X_2,X_3)=X_i$ for every $1\leq i\leq 3$.
Evaluating them at the points $\mathbf{z}_n'=(B_{2u_n}, A_{2u_n}, A_{2u_n-1})$, by \eqref{form1-palind}, \eqref{form1bis-palind} and \eqref{form2-palind} we get
\[\prod_{i=1}^3 |L_{i,p}'(\mathbf{z}_n')|_p |L_{i,\infty}'(\mathbf{z}_n')|_{\infty}\leq \frac{C_{\infty}^{6u_n}}{|B_{2u_n-1}^3|_p}\ll\left(\frac{C_{\infty}}{p^{k}}\right)^{6u_n}.\]
Moreover \[H(\mathbf{z}_n')\leq (C_{\infty} C_p)^{2 u_n}.\]
As by the definition of $k$ in Theorem \ref{main-thm-fiori} we have  $k> \frac{\log C_{\infty}}{\log p}$, Theorem \ref{thm:subspace-appl} ensures there exist $x_1,x_2,x_3\in \mathbb{Z}$ not all zero and an infinite subsequence of positive integers $n$ such that \begin{equation}\label{subspace-fiori-1}
    x_1 B_{2u_n}+x_2 A_{2u_n}+x_3 A_{2u_n-1}=0\end{equation} for all $n$.  
Dividing \eqref{subspace-fiori-1} by $B_{2u_n}$ we have 
\[x_1 +x_2 \frac{A_{2u_n}}{B_{2u_n}}+x_3 \frac{A_{2u_n-1}}{B_{2u_n-1}}\frac{B_{2u_n-1}}{B_{2u_n}}=0.\]
Since by \eqref{A_2un=B_2un-1} $A_{2u_n}=B_{2u_n-1}$, letting $n$ go to infinity, we get
\[x_1+x_2 \alpha+x_3 \alpha^2=0\]
contradicting the hypothesis that $\alpha$ is algebraic of degree at least $3$.

\subsubsection*{Case (2):  the set $\{n\geq 1\mid w_n=v_n=0\}$ is finite and the set $\{n\geq 1\mid w_n=0\}$ is infinite.}
Without loss of generality we might assume that, up to passing to a subsequence, for all $n$, $v_n\neq 0$, $w_n=0$ and that, by enlarging the words $U_n$, the first and the last letters of the word $V_n$ are always different (if not, it means that we are in Case (1)). So, in particular, the word $U_n V_n \widetilde{U_n}$ is never palindromic.

Notice first that, by Lemma~\ref{lem:palindromeCF}, \[\frac{B_{n-1}}{B_n}=[0, a_n,a_{n-1},\ldots,a_1]_s.\]
In particular 
\[\frac{B_{\ell_n+u_n-1}}{B_{\ell_n+u_n}}=[0,U_n\widetilde{V_n}\, \widetilde{U_n}]_s.\]
On the other hand 
\[\frac{A_{\ell_n+u_n}}{B_{\ell_n+u_n}}=[0,U_n{V_n}\widetilde{U_n}]_s.\]
Thus the continued fraction expansions of the numbers $\alpha$ and ${B_{\ell_n+u_n-1}}/{B_{\ell_n+u_n}}$ share the first $u_n$ terms, hence by Lemma \ref{lemma2-nostro} we have
\begin{equation}\label{condiz-bound-B_n-clubs}
\left|B_{\ell_n+u_n}\alpha-B_{\ell_n+u_n-1}\right|_p\leq |B_{\ell_n+u_n}|_p|B_{u_n}|_p^{-2}.
\end{equation}

We now apply Theorem \ref{thm:subspace-appl} with $N=4$ and the linear forms 
\begin{align*}
L_{1,p}(X_1,X_2,X_3,X_4)&=\alpha X_1-X_3, \\
L_{2,p}(X_1,X_2,X_3,X_4)&=\alpha X_2-X_4, \\
L_{3,p}(X_1,X_2,X_3,X_4)&=\alpha X_1-X_2, \\
L_{4,p}(X_1,X_2,X_3,X_4)&=X_2, 
\end{align*}
{and} 
\[
L_{i,\infty}(X_1,X_2,X_3,X_4)=X_i,
\]
for  $1\leq i\leq 4$. Evaluating them at the points \[\mathbf{z}_n=(B_{\sn}, B_{\sn-1}, A_{\sn}, A_{\sn-1}),\]  by hypothesis \ref{archim-A_n-e-B_n^1/n-is-bounded-clubs} of Theorem \ref{main-thm-fiori} and inequalities \eqref{eq-alpha-parti-quot}  and \eqref{condiz-bound-B_n-clubs}, we find that
\begin{equation*}
 \prod_{1\leq i\leq 4} |L_{i,p}(\mathbf{z}_n)|_{p} |L_{i,\infty}(\mathbf{z}_n)|_{\infty}\leq \frac{C_{\infty}^{4(\sn)}}{|B_{u_n}|_p^2}\leq \left(\frac{C_{\infty}^{4(2+c)}}{p^{2k}}\right)^{u_n}.
\end{equation*}
Notice now that \[H(\mathbf{z}_n)\leq (C_{\infty} C_p)^{(2+c)u_n}.\]
But then, since by the definition of $k$ in Theorem \ref{main-thm-fiori} \[k> (4+2c)\frac{\log C_{\infty}}{\log p},\]
Theorem \ref{thm:subspace-appl} ensures that there exist $x_1,x_2,x_3,x_4\in \mathbb{Z}$ not all zero and an infinite subsequence of positive integers $n$ such that, for all $n$, we have \begin{equation}\label{subspace-fiori-2-prima}
x_1 B_{\sn}+x_2 B_{\sn-1}+x_3 A_{\sn}+x_4 A_{\sn-1}=0.
\end{equation}
 Dividing by $B_{\sn}$ we obtain
\begin{equation}\label{subspace-fiori-2}
    x_1+x_2 \frac{B_{\sn-1}}{B_{\sn}}+x_3 \frac{A_{\sn}}{B_{\sn}}+x_4 \frac{A_{\sn-1}}{B_{\sn-1}}\cdot \frac{B_{\sn-1}}{B_{\sn}}=0.
\end{equation} Recalling that by \eqref{condiz-bound-B_n-clubs} we have
\begin{equation*}\label{condiz-bound-B_n-clubs-2}\lim_{n\rightarrow \infty}\frac{B_{\sn-1}}{B_{\sn}}=\alpha,
\end{equation*} 
letting $n$ tend to infinity in equation \eqref{subspace-fiori-2},  we get
\[x_1+(x_2+x_3)\alpha+x_4 \alpha^2=0.\]
Since $\alpha$ is not quadratic, we obtain $x_1=x_4=0$ and $x_2=-x_3$, hence, from \eqref{subspace-fiori-2-prima} we have \[A_{\sn}=B_{\sn-1}.\]
But then the matrix 
\[\begin{pmatrix}
B_{\sn}& B_{\sn-1}\\
A_{\sn}& A_{\sn-1}
\end{pmatrix}=
\begin{pmatrix}
a_{1}& 1\\
1& 0
\end{pmatrix}\begin{pmatrix}
a_{2}& 1\\
1& 0
\end{pmatrix}\cdots \begin{pmatrix}
a_{\sn}& 1\\
1& 0
\end{pmatrix}\]
is symmetric. Hence, by the uniqueness of the continued fraction expansion for a given floor function,   the word $a_1\cdots a_{\sn}=U_n V_n \widetilde{U_n}$ is palindromic, contradicting our assumption. 

\subsubsection*{Case (3): the set $\{n\geq 1\mid w_n= 0\}$ is finite.} 
Let
\begin{equation*}\label{A(),B()}A(t_{0},\ldots, t_m),\quad\quad B(t_{0},\dots, t_m)
\end{equation*}
be the functions defined by the continuants recurrences \eqref{def-continuants}, that is 
\[A({t_{0}})=t_0,\ A(t_0, t_1)=t_0t_1+1,\ A({t_{0}, \ldots, t_i})= t_{i} A(t_{0}, \ldots, t_{i-1})+A(t_{0}, \ldots, t_{i-2})\]
\[B({t_{0}})=1,\ B(t_{0}, t_1)=t_1,\ B({t_0, \ldots, t_i})= t_i B(t_0, \ldots, t_{i-1})+B(t_0, \ldots, t_{i-2}).\]

For any positive $n$, 
let 
\[P_n=A(0,W_n,U_n,V_n,\widetilde{U}_n,\widetilde{W}_n),\qquad Q_n=B(0,W_n,U_n,V_n,\widetilde{U}_n,\widetilde{W}_n)\]
and
\[P'_n=A(0,W_n,U_n,V_n,\widetilde{U}_n,\widetilde{W'_n})\qquad Q'_n=B(0,W_n,U_n,V_n,\widetilde{U}_n,\widetilde{W'_n})\]
where $W'_n$ denotes the word obtained by removing the first letter of $W_n$. We consider these as functions of $\ell_n+u_n+w_n$, respectively $\ell_n+u_n+w_n-1$, parameters.

Observe that
\begin{equation*}\label{eq:quasipal}
\frac{P_n}{Q_n}=[0,W_n,U_n,V_n,\widetilde{U}_n,\widetilde{W}_n]_s
\end{equation*}
and 
\begin{equation*}\label{eq:quasipal-primo}
\frac{P_n'}{Q_n'}=[0,W_n,U_n,V_n,\widetilde{U}_n,\widetilde{W'_n}]_s.
\end{equation*}
Notice that the continued fraction expansions of both $P_n/Q_n$ and $P_n'/Q_n'$ have the first $\ell_n+u_n$ partial quotients in common with $\alpha$. Therefore, by Section \ref{sect:continuants}, we have
\begin{align*}
%\label{c}
\begin{split}
    \begin{pmatrix}
        Q_n & Q'_n\\
        P_n & P'_n
    \end{pmatrix}
        &=\begin{pmatrix}
        B_{\ell_n+u_n} & B_{\ell_n+u_n -1}\\
        A_{\ell_n+u_n} & A_{\ell_n+u_n-1}
    \end{pmatrix}\begin{pmatrix}
        B_{w_n} & A_{w_n}\\
        B_{w_{n}-1} & A_{w_n-1}\end{pmatrix},
        \end{split}
\end{align*}
hence 
\begin{equation}
\label{c}
    \begin{split}
        Q_n=B_{\ell_n+u_n} B_{w_n}+ B_{\ell_n+u_n -1}B_{w_{n}-1},\\
Q_n'=B_{\ell_n+u_n} A_{w_n}+ B_{\ell_n+u_n -1}A_{w_{n}-1},\\
P_n=A_{\ell_n+u_n} B_{w_n}+ A_{\ell_n+u_n -1}B_{w_{n}-1},\\
P_n'=A_{\ell_n+u_n} A_{w_n}+ A_{\ell_n+u_n -1}A_{w_{n}-1}.
    \end{split}
\end{equation}
Again, since $\alpha$, $P_n/Q_n$ and $P'_n/Q'_n$ share the same first $\tn$ partial quotients, from Lemma~\ref{lemma2-nostro} we have
\begin{equation}
\label{eqn:ab10}
\left|Q_n\alpha - P_n\right|_p\leq \left|\frac{Q_n}{B_{\ell_n+u_n}^2}\right|_p\leq \left|\frac{B_{\ell_n+u_n} B_{w_n}}{B_{\ell_n+u_n}^2}\right|_p=\left|\frac{ B_{w_n}}{B_{\ell_n+u_n}}\right|_p,  
\end{equation}
\begin{equation}
\label{eqn:ab11}
|Q'_n\alpha-P'_n|_p\leq \left|\frac{B_{\ell_n+u_n} A_{w_n}}{B_{\ell_n+u_n}^2}\right|_p=\left|\frac{ A_{w_n}}{B_{\ell_n+u_n}}\right|_p\leq \left|\frac{ B_{w_n}}{B_{\ell_n+u_n}}\right|_p.  
\end{equation}
Notice also that, by Lemma~\ref{lem:palindromeCF},
    \[\frac{Q'_n}{Q_n}=[0,W_n,U_n,\widetilde{V_n},\widetilde{U_n},\widetilde{W_n}]_s.\]
    It follows from Lemma \ref{lemma2-nostro} that
\begin{equation}
    \label{eqn:ab12}
    |Q_n\alpha-Q'_n|_p<\left |\frac{Q_n}{B_{w_n+u_n}^2}\right |_p\leq \left |\frac{B_{\ell_n+u_n} B_{w_n}}{B_{w_n+u_n}^2}\right |_p
\end{equation}
showing, in particular, that
\begin{equation*}
%\label{eqn:limQnQ'n}
\lim_{n\rightarrow \infty} \frac{Q'_n}{Q_n}=\alpha.
\end{equation*}

Furthermore, by  \eqref{c},
\begin{equation}\label{eq:quattro}
|Q'_n|_p\leq |B_{\ell_n+u_n} A_{w_n}|_p.
\end{equation}

    We now apply Theorem \ref{thm:subspace-appl} with $N=4$ and  the linear forms 
\begin{align*}L_{1,p}''(X_1,X_2,X_3,X_4)=&\alpha X_1-X_3,\\
L_{2,p}''(X_1,X_2,X_3,X_4)=&\alpha X_2-X_4,\\
L_{3,p}''(X_1,X_2,X_3,X_4)=&\alpha X_1-X_2,\\
L_{4,p}''(X_1,X_2,X_3,X_4)=&X_2, \end{align*}
and  
\[L_{i,\infty}''(X_1,X_2,X_3,X_4)=X_i\]
for  $1\leq i\leq 4$. Evaluating such forms at the points \[\mathbf{z}_n''=(Q_n, Q'_n,P_n,P'_n)\]
it follows from \eqref{eqn:ab10},\eqref{eqn:ab11},\eqref{eqn:ab12} and \eqref{eq:quattro} that
\begin{equation*}\label{eq:stimap}
\prod_{1 \leq j\leq 4} |L_{i,p}''(\mathbf{z}_n'')|_p\leq \left |\frac  {B_{w_n}^4}{B_{w_n+u_n}^2}\right |_p.
\end{equation*}

Moreover, by~\eqref{c}, we have
\begin{equation}\label{eq:stimainfty} \max\left( |Q_n|_\infty,|Q'_n|_\infty,|P_n|_\infty,|P'_n|_\infty\right)\leq 2 C_\infty^{2\ell_n-v_n} \ll C_\infty^{(3c+2)u_n}\end{equation}
and \begin{equation}\label{eq:stima-p} \max\left( |Q_n|_p,|Q'_n|_p,|P_n|_p,|P'_n|_p\right) \ll C_p^{(3c+2)u_n}.\end{equation}

Since 
\begin{align*}\alpha^2 Q_n-\alpha Q'_n-\alpha P_n+P'_n &= \alpha(Q_n\alpha-P_n)-(Q'_n\alpha-P'_n)\\&= \alpha Q_n\left(\alpha-\frac {P_n}{Q_n}\right)-Q'_n\left(\alpha-\frac {P'_n}{Q'_n}\right)\\
&= (\alpha Q_n-Q'_n)\left(\alpha-\frac {P_n}{Q_n} \right)+Q'_n\left(\frac {P'_n}{Q'_n}-\frac {P_n}{Q_n}\right)\end{align*}
it follows from  \eqref{c}, \eqref{eqn:ab10}, \eqref{eqn:ab11} and \eqref{eqn:ab12} that
\begin{equation}\label{eq:Bugeaud5.4}\begin{split}
    |\alpha^2 Q_n-\alpha Q'_n-\alpha P_n+P'_n|_p &\leq \max\left( \left|\frac{B_{w_n}}{B_{w_n+u_n}^2B_{\ell_n+u_n}}\right |_p, \left |\frac{1}{B_{\ell_n+u_n}B_{w_n}}\right |_p\right) \\
    &=\left | \frac{1}{B_{\ell_n+u_n}B_{w_n}}\right |_p= \frac 1 {|Q_n|_p}.
\end{split}\end{equation}
We now also consider the linear form 
\[L_{5,p}''(X_1,X_2,X_3,X_4)=\alpha^2X_1-\alpha X_2-\alpha X_3+X_4.\]
From \eqref{eqn:ab11}, \eqref{eqn:ab12}, \eqref{eq:quattro}, and \eqref{eq:Bugeaud5.4}, we deduce that
\begin{equation}\label{eq:uffa2} 
\prod_{j=2}^5|L_{j,p}''(\mathbf{z}_n'')|_p \leq \left |\frac{B_{w_n}^2}{B_{w_n+u_n}^2}
\right |_p
\leq \frac 1 {\prod_{j=w_n+1}^{w_n+u_n}|a_j|^2} \leq \frac 1 {p^{2ku_n}}.
\end{equation}
From \eqref{c}, \eqref{eq:stimainfty}, \eqref{eq:stima-p} and \eqref{eq:uffa2} we get
\begin{equation*}\label{eq:uffa3} 
\prod_{j=1}^4|L_{j,\infty}''(\mathbf{z}_n'')|_\infty  \prod_{j=2}^5|L_{j,p}''(\mathbf{z}_n'')|_p  \leq \frac {2^4C_\infty^{4(2\ell_n-v_n)}}{p^{2ku_n}}
\leq 2^4 \left (\frac {C_\infty^{4(3c+2)}}{p^{2k}}\right)^{u_n}
\end{equation*}
and
\begin{equation*}H(\mathbf{z}_n'')  \ll %m_p(\mathbf{x}_n)m_\infty(\mathbf{x}_n)\leq 
(C_pC_\infty)^{(3c+2)u_n}. \end{equation*}
Therefore, since by the definition of $k$ in Theorem \ref{main-thm-fiori}
\begin{equation*}
%\label{eq:cappa}
k > (4+6c)\cdot \frac{\log C_\infty}{\log p},
\end{equation*}
  we conclude by Theorem~\ref{thm:subspace-appl} that there exist $t_1,t_2,t_3,t_4 \in \ZZ$ not all zero and an infinite subsequence of positive integers $n$ such that 
\begin{equation}\label{sottospazio-fiori-3}
    t_1 Q_n + t_2 Q'_n + t_3 P_n + t_4 P'_n=0
\end{equation}
for all $n$. Dividing by $Q_n$ and
noting that
\begin{equation}\label{limit-quotients-to-alpha}
    \lim_{n\to\infty}\frac {Q'_n}{Q_n}=\lim_{n\to\infty}\frac {P_n}{Q_n}=\lim_{n\to\infty}\frac {P'_n}{Q'_n}=\alpha,
\end{equation}
by letting $n$ go to infinity we obtain \[t_1+(t_2+t_3)\alpha+t_4\alpha^2=0.\]
Since $\alpha$ is not quadratic by hypothesis, we get $t_1=t_4=0$ and $t_2=-t_3$, which, replaced in \eqref{sottospazio-fiori-3}, imply $Q'_n=P_n$ for infinitely many $n$.

By~\eqref{eq:Bugeaud5.4}, we then have, for infinitely many such $n$, 
\begin{equation*}
    |\alpha^2 Q_n - \alpha Q'_n - \alpha P_n + P_n'|_p=|\alpha^2 Q_n -2 \alpha Q'_n  + P'_n|_p\ll \frac{1}{|Q_n|_p}.
\end{equation*}
We finally apply again Theorem \ref{thm:subspace-appl} with $N=3$ and the linear forms
\begin{align*}
    L_{1,p}'''(X_1,X_2,X_3)&=\alpha^2 X_1 -2\alpha X_2 + X_3\\
    L_{2,p}'''(X_1,X_2,X_3)&=\alpha X_2 - X_3\\
    L_{3,p}'''(X_1,X_2,X_3)&=X_1\\
    L_{i,\infty}'''(X_1,X_2,X_3)&= X_i
\end{align*}
for all $1\leq i\leq 3$.
By evaluating them  at the points
\[\mathbf{z}_n'''=(Q_n, P_n,P'_n)\]
we find
\begin{equation*}\label{eq:uffa4} 
\prod_{j=1}^3|L_{j,\infty}'''(\mathbf{z}_n''')|_\infty \cdot |L_{j,p}'''(\mathbf{z}_n''')|_p  \leq \frac {2^3C_\infty^{3(2\ell_n-v_n)}}{p^{2ku_n}}\leq 2^3 \left (\frac {C_\infty^{3(3c+2)}}{p^{2k}}\right)^{u_n}
\end{equation*}
and \[H(\mathbf{z}_n''')\ll (C_pC_\infty)^{(3c+2)u_n}.\]
Therefore, since the definition of $k$ in Theorem \ref{main-thm-fiori} implies
\[ k> \frac 3 2 (3c+2)\frac {\log C_\infty}{\log p}\]
and~\eqref{eq:stimainfty} holds, we conclude by  Theorem \ref{thm:subspace-appl} that there exist  $b_1,b_2,b_3,b_4\in \ZZ$ not all zero and an infinite subsequence of positive integers $n$ such that
\[b_1 Q_n + b_2 Q'_n + b_3 P'_n=0\]
for all $n$. Diving by $Q_n$, using \eqref{limit-quotients-to-alpha} and letting $n$ go to infinity we obtain $b_1+b_2\alpha+b_3\alpha^2=0$, contradicting the assumption that $\alpha$ is neither rational nor quadratic.
\qed

\section*{Acknowledgments}
The authors warmly thank the anonymous referee for their exceptionally careful reading of the manuscript and for their detailed comments, which have significantly improved the clarity, precision, and overall quality of this work. They are also deeply grateful to Yann Bugeaud for many enlightening discussions and for his precious comments on our work during his visit to the Dipartimento di Matematica G. Peano in Turin.

The authors also all acknowledge the IEA (International Emerging Action) project PAriAlPP (Problèmes sur l’Arithmétique et l’Algèbre des Petits Points) of the CNRS for the financial support.

Laura Capuano, Marzio Mula, and Lea Terracini are members of the INdAM group GNSAGA.

%\section*{To do}
 %   \textcolor{red}{Aggiungere indirizzi}\textcolor{red}{Proposte di riviste: Mathematical research letters (ma gli editor NT sembrano piu in geometria aritmetica (Ribet, Harris, Hesnault)); Manuscripta mathematica (editor TdN Philippe Gille - fa coomologia galoisiana)} \\

\appendix

\section{Pokering words: the \texorpdfstring{$\spadesuit$}{♠} and \texorpdfstring{$\clubsuit$}{♣} conditions}\label{section-2-picche-fiori-esempi}\label{app-A}

In this appendix we collect more results on $\spadesuit$ and $\clubsuit$ words and their combinatorial properties. In particular, we study how such words are related to other famous classes of words, such as automatic and Sturmian words. Moreover, we show how Corollary~\ref{corollario-automatiche} can be specialized to these classes of words. 

\subsection{Link to automatic words}
\label{subsec:automatic}
We briefly recall that, if $k\geq 2$ is an integer,  given a set $\mathcal{A}$, a sequence $(a_n)_{n \geq 1}$ with $ a_n\in \mathcal{A}$ is called \emph{$k$-automatic} if, for every $n$, the term $a_n$ can be computed by a finite state machine (or \emph{automaton}). This machine takes in input the $k$-ary digits of $n$, starting from an initial state. Each digit of $n$ corresponds to a state transition, and each state has an associated output. The output of the final state reached after reading all digits of $n$ determines the value of $a_n$. 
For further details on automata theory, we refer to \cite{alloucheShallit2003:automaticSequences}.

We can extend the notion of automaticity to words, by saying that a word $w=a_1a_2\cdots$ is \emph{$k$-automatic} if and only if the sequence $(a_n)_n$ is so.

As a finite automaton has only finitely many possible states and so finitely many possible outputs, any word taking on infinitely many different values is clearly not automatic. However, the finiteness of the values alone
is not enough to ensure automaticity. An example of a word on a finite alphabet which is not $k$-automatic for any $k\geq 2$ is the \emph{Fibonacci word}
\[\omega=01001010010010100101001010010100\cdots\]
whose letters are given by the sequence $(a_n)_{n\geq 1}$ where $a_n=2+\lfloor n\varphi \rfloor -\lfloor (n+1)\varphi \rfloor$
where $\varphi$ is the golden ratio. The Fibonacci word and, more generally, Sturmian words (see Section~\ref{sec:sturm}), are not $k$-automatic~\cite[§1]{lucasEtAl:cobham}.

A result by Cobham~\cite{cobham68:tagMachinesComplexity} shows that, for any automatic word $\omega$, there exists a constant $C>0$ such that $p_{\omega}(n)\leq Cn$ for all $n\geq 1$. This, together with Theorem \ref{bugeaud-complexity-picche} gives:
\begin{proposition}
Every automatic word satisfies condition $\spadesuit$.    
\end{proposition}
We now discuss some examples of automatic words.    
\subsubsection*{Thue-Morse words.}
One of the most classical and simplest examples of a $2$-automatic word is the Thue-Morse word mentioned earlier. Put $\mathcal{A}=\{a,b\}$. The \emph{Thue-Morse seqeunce on the alphabet $\mathcal{A}$} is the sequence $(a_n)_{n\geq 0}$ where $a_n=a$ if the number of $1$ in the binary writing of $n$ is even and $a_n=b$ otherwise. 
It is produced by the following automaton which has 2 states, where the arrows indicate the transition functions depending on the lecture of a binary digit of $n$.
\begin{center}
\begin{tikzpicture}[shorten >=1pt, node distance=2cm, on grid, auto] 
   \node (q_0)   {\boxed{a}}; 
   \node (q_1) [right=of q_0] {\boxed{b}}; 
    \path[->] 
    (q_0) edge[loop above] node {0} ()
          edge[bend left]  node {1} (q_1)
    (q_1) edge[loop above] node {0} ()
          edge[bend left]  node {1} (q_0);
\end{tikzpicture}
\end{center}
The corresponding \emph{Thue-Morse word} is
\[w=abbabaabbaababbabaababba\cdots.\] We have the following:
\begin{proposition}\label{prop:TM_is_spadesuit}
Thue-Morse words  satisfy:
\begin{enumerate}[label=(\roman*)]
\item\label{it1-Thue-Morse} condition $\spadesuit$ with constant $c=2$;
\item\label{it2-Thue-Morse} condition $\clubsuit$ with constant $c=0$.     
\end{enumerate}
\end{proposition}
\begin{proof}
Let $w$ be a Thue-Morse word.
    Notice that, for each $n$, any initial block of $\omega$ of length $2^n$ has the form $U_nV_nU_n$ with $|V_n|=2|U_n|=2^{n-1}$. This proves \ref{it1-Thue-Morse}. As for \ref{it2-Thue-Morse}, as remarked in \cite{adamBug2007:thueMorse}, denoting by $U_n$ the word consisting of the first $2^n$ letters of $\omega$, then $U_n\widetilde{U_n}$ is a prefix of $\omega$ for every $n$. Hence $\omega$ satisfies condition $\clubsuit$ with $|W_n|=|V_n|=0$.
\end{proof}
\subsubsection*{Rudin-Shapiro words.} Put $\mathcal{A}=\{a,b\}$. The \emph{Rudin-Shapiro sequence on the alphabet $\mathcal{A}$} is the sequence  $(a_n)_{n \geq 0}$ defined as: for every $n\geq 1$ in which case $a_n=a$ if the number of occurrences (possibly overlapping) of the block $11$ in the base-$2$ expansion of $n$ is even, or $a_n=b$ otherwise.

The corresponding \emph{Rudin-Shapiro word} is 
\[\omega=a a a b a a b a a a a b b b a b a a a b a a b a b b b a a a b a a a a b a a  \cdots\]
and it is $2$-automatic.
\begin{proposition}
\label{prop:RSconst}
    Rudin-Shapiro words satisfy condition $\spadesuit$ with constant $c=25$. 
\end{proposition}
\begin{proof}
    It has been proved that $p_{\omega}(n)=8(n-1)$ for every $n\geq 8$ (see  \cite{allouche-shallit-Rudin-shapiro} for a more general result). Thus the result follows from Theorem \ref{bugeaud-complexity-picche} with $C=8$.
\end{proof}
\subsubsection*{Paperfolding words.} Another classical example of $2$-automatic word is the \emph{paperfolding word}.
Put $\mathcal{A}=\{a,b\}$. The \emph{paperfolding sequence on the alphabet $\mathcal{A}$} is the sequence $(a_n)_n$ defined as it follows: for every $n\geq 1$, write $n=2^k m$ where $m$ is odd and set $a_n=a$ if $m\equiv 1 \pmod 4$ and $a_n=b$ otherwise. The corresponding \emph{paperfolding word} is \[w=a a b a a b b a a a b b a b b a a a b a a b b b a a b b a b b a a a b a a b b a a\cdots.\]  
\begin{proposition}\label{paper-folding}
    Paperfolding words satisfy condition $\spadesuit$ with constant $c=13$.
\end{proposition}
\begin{proof}
    By a result of Allouche (see \cite{Allouche-paper-folding-2}) $p_{\omega}(n)=4n$ for all $n\geq 7$. Hence, the result follows by Theorem \ref{bugeaud-complexity-picche} with $C=4$.
\end{proof}
\begin{remark}
    To the authors' knowledge, it is not known whether Rudin-Shapiro and paperfolding words satisfy condition $\clubsuit$. However, if this is the case we must have that the sequence $(V_n)_n$ defining the $\clubsuit$ condition can contain only finitely many empty words. Indeed, as proved in \cite{Allouche-paper-folding}, the longest palindromes in a Rudin-Shapiro word have length 14 while they have length 13 in a paperfolding word.
\end{remark}

\subsection{Link to Sturmian words.}
\label{sec:sturm}

A word $\omega$ is said to be \emph{Sturmian} if $p_{\omega}(n)=n+1$ for all $n\geq 1$. Equivalently (see for instance \cite[Section 4.2]{ALLOUCHE200139}) a word $w=a_0a_1a_2\cdots$ is Sturmian if and only if there exists a set $\mathcal{A}=\{a,b\}$ such that the sequence  $(a_n)_{n \geq 0}$ satisfies, for some irrational $\alpha \in (0, 1)$ and some real number $\beta$, one of the following properties
\begin{enumerate}[label=(\roman*)]
\item for all $n\geq 1$, $a_n = a$ if $ \lfloor (n+1)\alpha + \beta \rfloor - \lfloor n\alpha + \beta \rfloor=0$ and $a_n=b$ otherwise;
\item for all $n\geq 1$, $a_n = a$ if $\lceil (n+1)\alpha + \beta \rceil - \lceil n\alpha + \beta \rceil=0$ and $a_n=b$ otherwise.
\end{enumerate}

A word is called \emph{characteristic Sturmian} if it satisfies the above conditions with $\beta=0$. A classical example of characteristic Sturmian word is the Fibonacci word introduced above.

\begin{proposition}
\label{prop:sturmianConst}
We have that:
    \begin{enumerate}[label=(\roman*)]
        \item\label{it1-sturmian}  Sturmian words satisfy condition $\spadesuit$ with constant $c=0$;
        \item\label{it2-sturmian} Characteristic Sturmian words satisfy condition $\clubsuit$ with constant $c=0$.
        \item \label{it3-sturmian} Sturmian words satisfy condition $\clubsuit$ with constant $c=2$.
    \end{enumerate}
\end{proposition}
\begin{proof}
    Point \ref{it1-sturmian} follows from the fact that any Sturmian word begins with arbitrarily long squares, as proved in \cite[Proposition 2]{ALLOUCHE200139}.

    Point \ref{it2-sturmian} follows from the fact that characteristic Sturmian words starts with arbitrarily long palindromes (see for instance \cite{Fischler2005PalindromicPA}).

    Point \ref{it3-sturmian} requires some further work. Let $\alpha$ be a Sturmian word. Following~\cite[§1.3.4]{Lothaire_2002}, for each $n\geq 1$,  we define the \emph{factor graph} of $\alpha$ as the directed graph $G_n=G_n(\alpha)$ whose vertices are the distinct blocks of length $n$ appearing in $\alpha$ (there are $n+1$ distinct blocks in the case of Sturmian words), and such that there is a directed edge from a block $w_1 \cdots w_n$ to  $w'_1 \cdots w'_n$ if and only if $w_2\cdots w_n = w'_1\cdots w'_{n-1}$ and $w_1\cdots w_n w'_n$ is a block of $\alpha$ of length $n+1$. Notably, $\alpha$ is completely described by some path on $G_n(\alpha)$, namely the one starting on the vertex representing $a_1 \cdots a_n$, and such that the $i$-th step lands exactly the block corresponding to $a_{i+1} \cdots a_{n+i+1}$ for each $i\in \mathbb{N}$. 

By~\cite[63]{Lothaire_2002}, there exists a unique subword of length $n$, say $w_1 \cdots w_n$, called the \emph{$n$-th right special factor} of $\alpha$ and denoted by $R_n$, such that both $w_1\cdots w_na$ and $w_1\cdots w_n b$ are subwords of $\alpha$. Similarly, there exists a unique subword of length $n$, say $w'_1 \cdots w'_n$, called the \emph{$n$-th left special factor} of $\alpha$ and denoted by $L_n$, such that  both $aw_1\cdots w_n$ and $bw_1\cdots w_n$ are subwords of $\alpha$. In particular, for every $n\geq 1$, $R_n$ (resp.\ $L_n$) is the unique vertex of outdegree (resp.\ indegree) $2$ in the graph $G_n$.

Therefore, $G_n$ is as in the following picture (see also~\cite[82]{Lothaire_2002}).
\begin{center}
    \begin{tikzpicture}[>=stealth, node distance=1.5cm, thick]
    % Define nodes with reduced vertical distance
    \node (Ln) at (0, 0) {$L_n$};
    \node (Rn) at (0, 2.5) {$R_n$}; % Reduced height

    % Draw vertical arrows
    \draw[->] (Ln) -- (0, 0.7);
    \fill (0, 0.8) circle (1.5pt); % Dot at the start of the dotted line
    \draw[dotted,->] (0, 0.9) -- (0, 1.6);
    \fill (0, 1.7) circle (1.5pt); % Dot at the end of the dotted line
    \draw[->] (0, 1.8) -- (Rn);

    % Left loop with smaller size
    \draw[->] (Rn) .. controls (-0.8, 3) and (-0.9, 2.8) .. (-1.2, 2.6);
    \fill (-1.3, 2.5) circle (1.5pt); % Dot at the start of the dotted line
    \draw[dotted,->] (-1.35, 2.4) .. controls (-2, 1.6) and (-2, 0.6) .. (-1.2, -0.4);
    \fill (-1.1, -0.4) circle (1.5pt); % Dot at the end of the dotted line
    \draw[->] (-1.0, -0.4) .. controls (-0.7, -0.5) and (-0.4, -0.5) .. (Ln);

    % Right loop with smaller size
    \draw[->] (Rn) .. controls (0.6, 2.7) and (1.0, 2.4) .. (1.1, 2.2);
    \fill (1.1, 2.1) circle (1.5pt); % Dot at the start of the dotted line
    \draw[dotted,->] (1.1, 2.0) .. controls (1.2, 1.9) and (1.2, 1.1) .. (1.1, 1.0);
    \fill (1.1, 0.9) circle (1.5pt); % Dot at the end of the dotted line
    \draw[->] (1.1, 0.8) .. controls (1.1, 0.5) and (0.7, 0.2) .. (Ln);

\end{tikzpicture}
\end{center}

Moreover, one can check that $R_n = \widetilde{L_n}$. \\
To prove the $\clubsuit$ property, we claim that one can consider the initial block of length $4n+1$ and set $U_n = L_n$. Indeed, $L_n$ is reached at least once within the first  $n$ steps of any path on $G_n$. Next, to avoid overlapping, we still take $n-1$ steps on $G_n$. Starting from the reached vertex, $R_n$ (which is $\widetilde{L_n}$) is reached within $n+1$ steps. Thus a path of length $3n$ hits $L_n$ and $\widetilde{L_n}$ with no overlapping, which is equivalent to our claim.
\end{proof}

\subsection{Explicit estimates in Corollary~\ref{corollario-automatiche} for some families of words}\label{rem-appendix} We now show how condition \ref{itm:sublist} of Corollary~\ref{corollario-automatiche} can be made explicit, or even improved, for some  families of words  introduced in Sections~\ref{subsec:automatic} and~\ref{sec:sturm} on the alphabet $\{a,b\}\subset \mathbb{Z}[1/p]$. 
Suppose that \ref{itm:TM1} and \ref{itm:TM2} hold, i.e.\ $|a-b|_p\geq 1$ and $\min(|a|_p,|b|_p)\geq p$. Then:
%the constant $C$ in Corollary~\ref{corollario-automatiche} is explicit and hypothesis~\ref{itm:sublist} of Corollary~\ref{corollario-automatiche} can be in some cases relaxed. Namely,
\begin{itemize}
    \item If $\omega$ is a Rudin-Shapiro, \ref{itm:sublist} gives \[\left(\max\left(|a|_\infty, |b|_\infty\right)+1\right)^{152}<p\] (taking $C=8$ in Corollary~\ref{corollario-automatiche}, as computed in the proof of Proposition~\ref{paper-folding}).
    \item If $\omega$ is a paperfolding word, 
    \ref{itm:sublist} gives \[\left(\max\left(|a|_\infty, |b|_\infty\right)+1\right)^{80}<p\]
    (taking $C=4$ in Corollary~\ref{corollario-automatiche} as computed in the proof of Proposition~\ref{prop:RSconst} and Proposition~\ref{paper-folding}).
    \item If $\omega$ is Sturmian, \ref{itm:sublist} can be improved as  \[\left(\max\left(|a|_\infty, |b|_\infty\right)+1\right)^{3}<p\]
     This is in fact equivalent to~\eqref{eqn:AssFiniteAlph1} from Corollary~\ref{cor:fin-alph} with $c=0$ (the latter constant being deduced from  Proposition~\ref{prop:sturmianConst}.~\ref{it2-sturmian}).
    \item If $\omega$ is Thue-Morse or characteristic Sturmian, then one can also improve~\ref{itm:sublist} as \[\left(\max\left(|a|_\infty, |b|_\infty\right)+1\right)<p.\] This is in fact equivalent to \eqref{eqn:AssFiniteAlph2} from Corollary~\ref{cor:fin-alph} with $c=0$ (the latter constant is deduced, respectively, from Proposition~\ref{prop:TM_is_spadesuit}.~\ref{it2-Thue-Morse}  and Proposition~\ref{prop:sturmianConst}~\ref{it2-sturmian}).
\end{itemize}

\subsection{Some negative examples}
   As for words that do \emph{not} satisfy condition $\spadesuit$ or $\clubsuit$, an obvious example is given by those having all distinct letters from an infinite alphabet. However, words containing only two letters might also fail to satisfy these conditions. An example is the word $\alpha$ obtained adjoining the words
    \[w_i=\underbrace{0\cdots 0}_{\text{$i$ times}}\underbrace{1\cdots 1}_{\text{$i$ times}}.\]
Given any finite prefix $\beta$ of $\alpha$, let $j$ be the largest index such that $w_j$ is contained in $\beta$. Then, the longest repeated subword of $\beta$ is $w_{j-1}$. Equivalently, the largest possible choice for $U_j$ in $\beta$ would have length $2(j-1)$. However, even in this case
\begin{equation*}
    |W_j|=\sum_{k=1}^{j-2} 2k=(j-2)(j-1),
\end{equation*}
so that the ratio $|W_n|/|U_n|$ tends to infinity. A similar argument shows that $\alpha$ also fails to satisfy condition $\clubsuit$.

The periodic word $\overline{101001000}$ is trivially $\spadesuit$ but not $\clubsuit$, as the longest palindrome it contains is $0100010$.

Finally, we construct a word that is $\clubsuit$ but not $\spadesuit$. Defining the finite words
\[\beta_i = 0 \underbrace{1 \cdots 1}_{i \text{ times}},\]
we claim that
\[\alpha= \beta_1 \widetilde{\beta_1}\beta_2 \beta_3\widetilde{\beta_3}\widetilde{\beta_2}\cdots  \beta_{2^n} \cdots \beta_{2^{n+1}-1} \widetilde{\beta_{2^{n+1}-1}}\cdots \widetilde{ \beta_{2^n}} \cdots\]
is the word we are looking for. Indeed, it is $\clubsuit$ by setting $U_n=\beta_{2^{n}} \cdots \beta_{2^{n+1}-1}$ and $V_n=\emptyset$, so that \[|U_n|=\sum_{i=2^n+1}^{2^{n+1}} i=\frac{(2^n +1+ 2^{n+1})(2^{n+1}-2^n)}{2}\geq 2^{2n}\] and
\[|W_n| =(2+2^n)(2^n-1)\leq 2^{2n+1}, \]
proving that $\alpha$ is $\clubsuit$ with constant $2$. To conclude, suppose by contradiction that $\alpha$ were $\spadesuit$, and let $(U_n)_n,(V_n)_n$, and $(W_n)_n$ be as in Definition~\ref{def-spade-club-with-constant}. From the structure of $\alpha$ one can easily deduce that $U_n$ contains at most two zeroes. Let $k_n$ be the maximum number of consecutive $1$-s in $U_n$, so that $|U_n|\leq 3k_n+2$. Then $|W_n|\geq (2+2^{k_n-2})(2^{k_n-2}-1)\geq 2^{2(k_n-2)}$. Since $(k_n)_n$ must be unbounded, also $(|W_n|/|U_n|)_n$ is, contradicting part~\ref{it2} of Definition~\ref{def-spade-club-with-constant}.

\subsection{Some more remarks on condition \texorpdfstring{$\spadesuit$}{♠}}
Words satisfying $\spadesuit$ condition are sometimes called \emph{stammering words} (see \cite[Paragraph after Theorem 1.3]{Bugeaud2010AutomaticCF}). Indeed, they are an analogue (and, to some extent, a generalisation) of the stammering words defined in \cite{Adam-Bug-Annals-2007}, which are words beginning with arbitrarily large repetitions. Following \cite{Adam-Bug-Annals-2007}, a word $\omega$ is called \emph{stammering} if there exists a real number $w>1$ and two sequences $(U_n)_n, (W_n)_n$ of finite words and  such that conditions \ref{it2} and \ref{it3} hold and condition \ref{it1} is replaced by requiring that, for every $n\geq 1$, the word $W_n U_n^{w}$ is a prefix of $\omega$ (here $U^{w}$ denotes the words obtained repeating $\lfloor w\rfloor$ times the word $U$, followed by the prefix of $U$ of length $\lceil(w-\lfloor w\rfloor)|U|\rceil$.
In particular, $\spadesuit$ words having $|V_n|=0$ for all $n\geq 1$ are stammering words with $w=2$. On the other hand, $\spadesuit$ words are a generalisation of the stammering ones for which $w$ is an even integer. Indeed, if $\omega$ is stammering with $w=2 w'$ for some $w'\in \mathbb{N}$, then there are sequences $W_n$ and $U_n'$ of finite words such that $W_n U_n'^{2w}$ is a prefix of $\omega$ for every $n$, hence $\omega$ satisfies condition $\spadesuit$ with $V_n$ the empty word and $U_n=U_n'^w$.

\subsection{Some more remarks on condition \texorpdfstring{$\clubsuit$}{♣}}
As we have seen, there are only few examples and results on words satisfying condition $\clubsuit$.

However, in  \cite{AdamczewskiBugeaud2007:palindromic} the authors proposed a general construction for the class of words satisfying condition $\clubsuit$ with constant $c=0$ i.e.\ sequences beginning with arbitrarily large palindromes.

Given an arbitrary sequence $(R_n)_{n\geq 0}$ of nonempty finite words whose letters are positive integers, they define a sequence of finite words $(T_n)_{n\geq 0}$ by setting $T_0 = R_0$ and 
\[
T_{n+1} = T_n R_{n+1} \widetilde{T_n R_{n+1}},
\]
for $n \geq 0$. Thus, $T_{n+1}$ begins with $T_n$, and the sequence of finite words $(T_n)_{n\geq 0}$ converges to an infinite word $\omega$ such that, if $U_n=T_n R_{n+1}$, then $U_n \widetilde{U_n}$ is a prefix of $\omega$ for every $n\geq 0$ and $|U_n|$ tends to infinity with $n$. 
They also remark that such a construction actually gives every sequence beginning with arbitrarily large palindromes.

We end this section with the following Figure~\ref{fig:words}, which illustrates the relations between the families of words we have considered so far.

\begin{figure}[h]
    \centering
    \begin{tikzpicture}[scale=0.75] % Riduce la figura di 1/3
    % Ellipse finite
    \draw[line width=0.25mm, rotate=-45] (2,0) ellipse (4.5cm and 3cm);
    \node at (4.5,-4.7) [font=\scriptsize, fill=white] {\begin{tabular}{c}
             On a finite\\
          alphabet
        \end{tabular}};

    % Ellipse spade
    \draw[line width=0.25mm, rotate=45] (-2,0) ellipse (4.5cm and 3cm);
    \node[fill=white, font=\scriptsize] at (-4.5,-4.7) {$\spadesuit$};

  % Ellipse fiori
    \draw[line width=0.25mm] (0,1.7) ellipse (4.5cm and 1.7cm);
    \node[fill=white, font=\scriptsize] at (0,3.57) {$\clubsuit$};

    % Automatic words
    \draw (0,-1) circle (1.3) node [below=0.2cm, font=\scriptsize] {Automatic};

    % Periodic (modificato in un ovale e ridotto)
    \draw (0,-0.55) ellipse (0.75cm and 0.75cm) node[font=\scriptsize] {Periodic};

    % Sturmian
    \draw (0,1.35) circle (0.8) node [font=\scriptsize] {Sturmian};

    % Ellipse low-complexity
    \draw[] (0,-0.7) ellipse (1.6cm and 3cm);
    \node at (0,-2.8) [font=\scriptsize] {\begin{tabular}{c}
             Low-\\
          complexity
        \end{tabular}};

    \end{tikzpicture}
    \caption{Families of special words}
    \label{fig:words}
\end{figure}
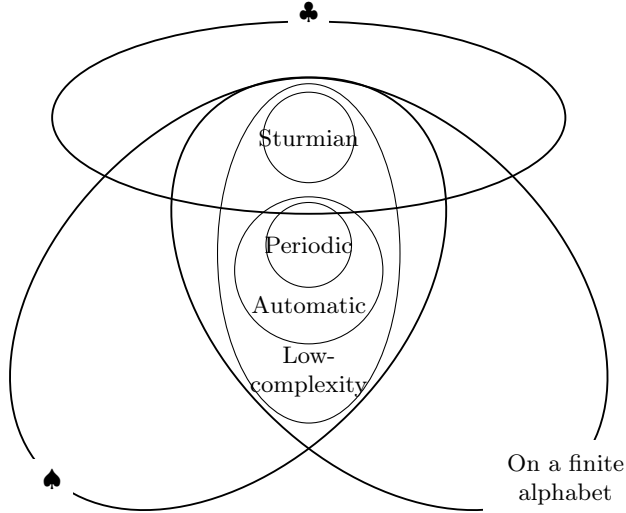

\printbibliography

\end{document}